\documentclass{amsart}
\usepackage{etex}
\usepackage{amsthm,stmaryrd}
\usepackage{amssymb}
\usepackage{epsfig}
\usepackage[usenames,dvipsnames]{color}
\usepackage{verbatim}
\usepackage{hyperref}
\usepackage{mathrsfs}
\usepackage[all]{xy}
\usepackage{xcolor}
\usepackage{enumerate}

\usepackage{pst-node}
\usepackage{tikz-cd}

\newcommand{\End}{\operatorname{End}}
\newcommand{\Hom}{\operatorname{Hom}}

\newcommand{\Negl}{\operatorname{Negl}}
\newcommand{\Id}{\operatorname{Id}}

\newcommand{\Z}{\mathbb{Z}}

\newcommand{\R}{\mathbb{R}}
\newcommand{\C}{\mathbb{C}}
\newcommand{\N}{\mathbb{N}}
\newcommand{\sll}{\mathfrak{sl}}

\newcommand{\p}[1]{\ensuremath{|{#1}|}}

\newcommand{\slto}{{\mathfrak{sl}(2|1)}}
\newcommand{\Uq}{{U_q(\slto)}}
\newcommand{\UqH}{{U_q^H(\slto)}}
\newcommand{\UqHa}{{U_q^H}}
\newcommand{\cat}{\mathscr{C}}

\newcommand{\catd}{\mathscr{D}}
\newcommand{\catdH}{\mathscr{D}^H}
\newcommand{\Gr}{G}
\newcommand{\X}{\ensuremath{\mathcal{X}}}
\newcommand{\XX}{{\ensuremath{\mathcal{Y}}}}

\newcommand{\qd}{\operatorname{\mathsf{d}}}
\newcommand{\rt}{l}
\newcommand{\A}{\ensuremath{\mathsf{A}}}
\newcommand{\bb}{\operatorname{\mathsf{b}}}
\newcommand{\ang}[1]{{\left\langle{#1}\right\rangle}}

\newcommand{\unit}{\ensuremath{\mathbb{I}}}

\newcommand{\tcoev}{\stackrel{\longleftarrow}{\operatorname{coev}}}
\newcommand{\tev}{\stackrel{\longleftarrow}{\operatorname{ev}}}
\newcommand{\ev}{\stackrel{\longrightarrow}{\operatorname{ev}}}
\newcommand{\coev}{\stackrel{\longrightarrow}{\operatorname{coev}}}
\newcommand{\T}{{\mathcal{T}}}
\newcommand{\D}{{\mathcal{D}}}
\newcommand{\FK}{{\Bbbk}}
\newcommand{\ideal}{\mathcal{I}} 
\newcommand{\qt}{\operatorname{\mathsf{t}}}
\newcommand{\qtN}{\operatorname{\mathsf{t}^N}}
\newcommand{\qdN}{\operatorname{\mathsf{d}^N}}

\newcommand{\bp}[1]{{\left(#1\right)}}

\newcommand{\Ffun}{\ensuremath{\mathsf{F}}}

\newcommand{\LL}{\mathcal{L}}

\newcommand{\I}{\mathcal{I}}
\newcommand{\qdim}{\operatorname{qdim}}
\renewcommand{\wp}{{\Phi}}
\newcommand{\B}{{\mathsf{B}}}

\newtheorem{definition}{Definition}[subsection]
\newtheorem{theorem}[definition]{Theorem}

\newtheorem{proposition}[definition]{Proposition}
\newtheorem{lemma}[definition]{Lemma}
\newtheorem{remark}[definition]{Remark}

\newtheorem{corollary}[definition]{Corollary}
\newcounter{exo} \newcounter{numexercice}
\renewcommand{\theexo}{\arabic{exo}}

\newcounter{IntroCounter}
\stepcounter{IntroCounter}

\begin{document}
\title[Modified Turaev-Viro Invariants from quantum $\slto$]{Modified Turaev-Viro Invariants\\ from quantum $\slto$} 
  \author{Cristina Ana-Maria Anghel}
\address{Univ Paris Diderot, Sorbonne Paris Cit\'e, IMJ-PRG, UMR 7586 CNRS, F-75013, Paris, France} \email{cristina.anghel@imj-prg.fr; simple\_words91@yahoo.com} 
\author{Nathan Geer}
\address{Mathematics \& Statistics\\
  Utah State University \\
  Logan, Utah 84322, USA} \email{nathan.geer@gmail.com} 
  \thanks{ This project began at a workshop hosted by Anna Beliakova with the support of NCCR SwissMAP founded by the Swiss National Science Foundation.    Research of NG was partially supported by NSF grants DMS-1308196 and DMS-1452093.  He would also like to thank the CRNS for its support.  
 The research of CA was supported by grants from R\'egion Ile-de-France.  
She would also like to thank the Isaac Newton Institute for Mathematical Sciences, Cambridge, for support and hospitality during the programme 
Homology theories in low dimensional topology where work on this paper was undertaken between 12 February-28 March 2017. During this time CA was also supported by EPSRC grant no EP/K032208/1.
Finally, we would like to thank Bertrand Patureau-Mirand for his insights.   
  }
  \date{\today}

\begin{abstract}
The category of finite dimensional module over the quantum superalgebra $\Uq$ is not semi-simple and the quantum dimension  of a generic $\Uq$-module vanishes.  This vanishing happens for any value of $q$ (even when $q$ is not a root of unity).    These properties make it difficult to create a fusion or modular category.   Loosely speaking, the standard way to obtain such a category from a quantum group is: 1) specialize $q$ to a root of unity; this forces some modules to have zero quantum dimension, 2) quotient by morphisms of modules with zero quantum dimension, 3) show the resulting category is finite and semi-simple.  In this paper we show an analogous construction works in the context of $\Uq$ by replacing the vanishing quantum dimension with a modified quantum dimension.  In particular, we specialize $q$ to a root of unity, quotient by morphisms of modules with zero modified quantum dimension and show the resulting category is generically finite semi-simple.   Moreover, we show the categories of this paper are relative $\Gr$-spherical categories.  As a  consequence we obtain invariants of 3-manifold with additional structures.  
\end{abstract}

\maketitle
\setcounter{tocdepth}{1}

\section{Introduction} 
The numerical $6j$-symbols associated with the  Lie algebra
$\sll_2(\C)$ were first introduced in theoretical physics by Eugene
Wigner in 1940 and Giulio (Yoel) Racah in 1942.  These symbols have been generalized to quantum $6j$-symbols coming from tensor categories.   If the category is fusion and spherical then the quantum $6j$-symbols
lead to Turaev-Viro invariants of 3-manifold (see \cite{BW,BW+,KRR,TV,Tu}).  
The prototype of such a topological invariant arises from a particular category of modules over the quantum algebra $U_q(\sll_2(\C))$.

Let us describe the T-V construction for this example.  
  Without modification the category of finite dimensional modules over $U_q(\sll_2(\C))$ is not fusion.  If $q$ is generic then there are an infinite number of non-isomorphic simple modules.  When $q$ is a root of unity then the quantum dimension of some of these modules becomes zero.  Loosely speaking, by taking the quotient of such modules one obtains a category with a finite number of simple modules.  More precisely,   
taking the quotient of the category of $U_q(\sll_2(\C))$-modules by negligible morphisms one obtains the desired spherical and fusion category $\mathscr{S}$.     
Here a morphism $f:V\to W$ is negligible if for all morphisms $g:W\to V$ we have
$$Tr_q(f\circ g)=0$$
where $Tr_q$ is the quantum trace.   If $V$ is a simple module whose quantum dimension $\text{qdim}(V)=Tr_q(\Id_V)$ is zero then any morphism to or from $V$ is negligible; such a module is called negligible.  The simple modules which are not negligible are said to be in the alcove.  

The Turaev-Viro invariant is defined as a certain
state sum computed on an arbitrary triangulation of a $3$-manifold.
 The state sum on a
triangulation $\T$ of a closed $3$-manifold $M$ is defined, roughly
speaking, as follows:   Consider states of $\T$ which are maps 
from the edges of $\T$ to a finite index set $I$ corresponding to
isomorphism class of simple objects in the category $\mathscr{S}$.   Given a state $\varphi:\{\text{edges of $\T$}\}\to I$ 
one associates with each tetrahedron $T$ of $\T$ a particular
quantum $6j$-symbol denoted by $|T|_{\varphi}$.  The state sum is defined by taking the product of these
symbols over all tetrahedra of $\T$ and summing up the resulting
products (with certain weights) over all $I$ colorings of $\T$:
 \begin{equation}\label{E:DefUsualTV}
TV(\T)=\D^{-2v}
 \sum_{\varphi\text{ state }}
   \left( \prod_{e\in\{\text{edges of } \T\}}
{\text{qdim}(\varphi(e))} \right)\left(\prod_{T\in\{\text{tetrahedron of } \T\}}{
|T|_{\varphi}}\right)
 \end{equation}
where $\text{qdim}(\varphi(e))$ is the quantum dimension of the simple module associated to $\varphi(e)$ and $\D^2=\sum_{i\in I}\text{qdim}(i)^2$.  

The main point of this construction   is that the state sum $TV(\T)$ is
independent of the choice of triangulation.   This can be verified
in two steps.  First, the quantum $6j$-symbols satisfy the symmetries of the tetrahedron.  Second, any two triangulations of a closed $3$-manifold can be
transformed into one another by a finite sequence of the so called
Pachner moves and an ambient isotopy (see \cite{P}).
 Thus, it is enough to check that the state sum is invariant under the Pachner moves.  For the category $\mathscr{S}$, these moves correspond to well-known algebraic identities which the quantum $6j$-symbols satisfy.

Obstructions to applying this construction to a general pivotal tensor category $\cat$ include: 
\begin{enumerate}
\item \label{I:ObPiv1}  zero quantum dimensions, 
\item  \label{I:ObPiv3} non-semi-simplicity of $\cat$, 
\item  \label{I:ObPiv4} infinitely many isomorphism classes of simple objects of $\cat$. 
\end{enumerate}
   Kashaev \cite{Kv94} and later Baseilhac and Benedetti
\cite{BB} considered $3$-manifold invariants arising from a 
category with such obstructions, namely the category of modules over
the Borel subalgebra of quantum $\sll(2)$ at a root of unity.  
Geer, Patureau-Mirand, and Turaev \cite{GPT2} gave an alternate general approach to dealing with these obstructions and defined a secondary Turaev-Viro invariant of oriented $3$-manifolds $M$.  This is accomplished by three main modifications of the T-V invariant.  

First, to address obstruction (\ref{I:ObPiv1}) they replace the  vanishing $\text{qdim}$ and $|T|_{\varphi}$ in Equation \eqref{E:DefUsualTV} with corresponding non-zero modified quantum dimension and modified $6j$-symbol, see \cite{GPT,GPT2}.  

The second modification is dealing with 
the last two obstructions.  
If the usual state sum described above is applied to a category with infinitely many isomorphism classes of simple objects, this sum is of course infinite.  With this in mind, the authors of \cite{GPT2} required that the pivotal tensor category $\cat$ had additional structure, including  a $\Gr$-grading on $\cat$ where $\Gr$ is an abelian group.   To overcome the infinite sum problem, a finite number of modules are selected using a cohomology class in $H^1(M,\Gr)$.  This step also addresses obstruction (\ref{I:ObPiv3}) by requiring that generically graded pieces of the category are semi-simple.   

The final modification is to introduce a link $L$ in the manifold $M$.   If one applies the changes described above, in the first two steps the invariant can still be zero or not well defined (see \cite{GPT2}).  In particular, the sum of the modified dimensions over a graded piece is zero, i.e. the analogous quantity associated to $\D^2$ in Equation \eqref{E:DefUsualTV} is zero.  To construct a non-zero invariant one can consider triangulations $\T$ of $M$ that realized the isotopy class of $L$ as a so called Hamiltonian link in $\T$ (see \cite{BB}).     The Hamiltonian link is used to modify the weights in the terms of the state sum.  

Let us be more precise.   The construction of the modified TV-invariant works in the context of a relative $\Gr$-spherical category (see Section \ref{S:CatPrelim} for details):  Let $\Gr$ be an abelian group with a small symmetric subset $\X\subset\Gr$.  Let $\cat= \bigoplus_{g\in\Gr}\cat_g$ be a $\Gr$-graded pivotal $\FK$-category such that 
\begin{enumerate}
  \item $\cat_g$ is finitely
    semi-simple for each
    $g\in\Gr\setminus\X$, 
   \item there exists a  t-ambi pair $(\A,\qd:\A\rightarrow
    \FK^{\times})$ where 
   each object of $\A$ is isomorphic to a unique element of  $\cup_{g\in\Gr\setminus \X}\I_g$   and $\I_g$ represents the isomorphic  classes of simple objects of $\cat_g$,
  \item there exists a map $\bb:\A\to \FK$ satisfying the condition in Definition \ref{D:G-spherical}.
\end{enumerate}
In \cite{GPT2} it is shown that a relative $\Gr$-spherical category (with basic data) gives rise to modified quantum $6j$-symbols.  In this
context, these $6j$-symbols are not numbers but rather tensors on certain  multiplicity spaces.   

Let $M$ be a closed orientable
$3$-manifold and $L$ a link in $M$. Following \cite{BB}, we consider  \emph{$H$-triangulation} $(\T,\LL)$ of $(M,L)$: $\T$ is a quasi-regular triangulation of $M$, $\LL$ is a set of unoriented edges of $\T$ such that every
vertex of $\T$ belongs to exactly two edges of $\LL$ and the union of the edges in $\LL$ is the link $L$.  Let $\wp:\{\text{edges}\}\to \Gr$ be a $\Gr$-valued $1$-cocycle on $\T$ which takes values in $\Gr\setminus \X$.  A {\it state} of the $\Gr$-coloring $\wp$ is a map $\varphi$ assigning to every
oriented edge $e$ of $\T$ an element $\varphi (e)$ of $\I_{\wp(e)}$ such that
$\varphi(-e)=\varphi (e)^*$ for all $e$. 
The set of all states of $\wp$ is denoted
$\operatorname{St}(\wp)$. 

Give a state $\varphi$ and a
tetrahedron $T$ of $\T$ we can associate a modified $6j$-symbols $|T|_\varphi$, for details see \cite{GPT2}.  Any face of $\T$ belongs to exactly
two tetrahedra of $\T$, and the associated multiplicity modules are dual to
each other.  The tensor product of the $6j$-symbols $|T|_\varphi$ associated to all
tetrahedra $T$ of $\T$ can be contracted using this duality.  We denote by
$\operatorname {cntr}$ the tensor product of all these contractions.  Let $\T_1$ be the set
of unoriented edges $\T$ and let $\T_3$ the set of tetrahedra of $\T$. Set
$$
TV(\T,\LL,\wp)=\sum_{\varphi\in \operatorname{St} ( \wp)}\,\, \prod_{e\in\T_1{\setminus \LL}}
\qd(\varphi(e))\, \prod_{e\in \LL} \bb(\varphi(e))\, \, \operatorname {cntr} \left
  (\bigotimes_{T\in\T_3}|T|_{\varphi}\right ) \in \FK.
$$

\begin{theorem}\label{inveee} $
  TV(\T,\LL,\wp)$ depends only on the isotopy class of $L$ in $M$ and the
  cohomology class $[\wp]\in H^1(M,G)$. It does not depend on the choice of
  the $H$-triangulation of $(M,L)$ and on the choice of $\wp$ in its
  cohomology class.
\end{theorem}


In this paper we show that the category $\catd$ of finite dimensional modules over the quantum linear Lie superalgebra $U_q(\slto)$ leads to a relative $\Gr$-spherical category and so gives rise to a modified TV-invariant.   Our approach is based on a generalization of the usual technique: we take a quotient by negligible morphisms corresponding to the \emph{modified trace}.  As we will discuss, the obtained invariants should have different properties than the usual TV-invariants.  

The category $\catd$ has the obstructions \eqref{I:ObPiv1}--\eqref{I:ObPiv4} listed above.  When $q$ is generic, the simple $U_q(\slto)$-modules are indexed by pairs $(n,\alpha)$ where $n\in \N$ and $\alpha\in \C$.   Even when $q$ is generic most of these simple modules have vanishing quantum dimensions.  Therefore, taking the quotient of $\catd$ by the negligible morphisms corresponding to the quantum trace would trivialize most of the category.  Alternatively, the ideal $\I$ generated by the four dimensional module $V(0,\alpha)$ has a non-zero modified trace and corresponding non-zero modified quantum dimensions.  When $q=e^{2 \pi i/\rt}$ is a root of unity some of these modified quantum dimensions become zero.  In particular, the modified quantum dimensions of modules on the boundary of the ``alcove'' vanish.    The idea behind this paper is to take the quotient of $\catd$ by these modules to obtain a relative $\Gr$-spherical category.  

This construction has a novel feature:   The alcove has an  infinite number of  
non-isomorphic 
simple modules $V(n,\tilde{\alpha})$ where $\tilde{\alpha}\in \C/\rt\Z$, $0\leq n\leq\rt'-2$ and $\rt'=\frac{2\rt}{3+(-1)^\rt}$.   Let $V(n,\alpha)$ be the simple module over  $U_q(\slto)$ when $q$ is generic.  By setting $q=e^{2 \pi i/\rt}$ the module $V(n,\alpha)$  specializes to $V(n,\alpha + \rt\Z)$ for $\alpha \in \C$.   Modified $6j$-symbols corresponding to these modules exist for both $q$ generic and $q=e^{2 \pi i/\rt}$.   An interesting problem is to see if the $6j$-symbols coming from generic $q$ will specialize to the $6j$-symbols constructed in this paper for  roots of unity of \textbf{any} order.  If this is true one could see if there exist an invariant defined for generic $q$ which when specialized to  $q=e^{2 \pi i/\rt}$, for any $\rt$ recovers the modified TV-invariants corresponding to the relative $\Gr$-spherical categories defined in this paper.   Results in this direction have been obtain for the usual Reshetikhin-Turaev quantum invariants, see \cite{Hab2008,HabLe2016,Le2003,Oht1995} and references within.

What is different in our context is that the modified invariants arising from the work of this paper require manifolds with additional structures, including the cohomology class discussed above.  This makes the definition of the invariants more technical but can also add new information.  For example, the modified RT-type invariants of \cite{CGP,BCGP} recover the multivariable Alexander polynomial and Reidemeister torsion, which allows the reproduction of the classification of lens spaces.  Also, the invariants of \cite{CGP,BCGP} give rise to TQFTs and mapping class group representations with the notable property that the action of a Dehn twist has infinite order.  This is in strong contrast with the usual quantum mapping class group representations where all Dehn twists have finite order. 
 We expect similar properties for the invariants coming from this paper.   
 Combining such properties with the ideas discussed in the previous paragraph could lead to appealing applications.

We should also mention that 3-manifold invariants arising from a quantization of $\slto$ have already been constructed by Ha in \cite{Ha}.  Ha's construction uses a ``unrolled'' version of quantum $\slto$.  He also uses a modified trace on a 
different ideal which consists of projective modules, most of which only exist when $q$ is a root of unity.   
 In comparison, in this paper the quotient of $\catd$ we take contains all the projective modules of $\catd$.  Also, for $0\leq n\leq\rt'-2$ and $\alpha\in \C$ the ideal Ha works with does not contain a module which is the specialization of the simple $U_q(\slto)$-module  $V(n,\alpha)$ discussed above.

\section{Categorical Preliminaries} \label{S:CatPrelim}

As mentioned above our main theorem will be that $\Uq$ gives rises to a  relative $\Gr$-spherical category.    With this in mind, in this section we will recall the general definition of such a category, for more details see \cite{GP2, GPT2}.    

\subsection{$\FK$-categories}
Let $\FK$ be a field.  
A \emph{tensor $\FK$-category} is a tensor category $\cat$ such that its hom-sets are left $\FK$-modules, the composition and tensor product of morphisms is $\FK$-bilinear, and the canonical $\FK$-algebra map $\FK \to \End_\cat(\unit), k \mapsto k \, \Id_\unit$ is an isomorphism (where $\unit$ is the unit object).
A tensor category is \emph{pivotal} if it has dual objects and duality morphisms 
$$\coev_{V} : \:\:
\unit \rightarrow V\otimes V^{*} , \quad \ev_{V}: \:\: V^*\otimes V\rightarrow
\unit , \quad \tcoev_V: \:\: \unit \rightarrow V^{*}\otimes V\quad {\rm {and}}
\quad \tev_V: \:\: V\otimes V^*\rightarrow \unit$$ which satisfy compatibility
conditions (see for example \cite{BW, GKP2}).  

 An object $V$ of $\cat$ is \emph{simple} if $\End_\cat(V)= \FK \Id_V$.
Let $V$ be an object in $\cat$ and let $\alpha:V\to W$ and $\beta:W\to V$ be morphisms.  The triple
$(V,\alpha,\beta)$ (or just the object $V$) is a \emph{retract} of $W$ if
$\beta\alpha=\Id_{V}$.  An object $W$ is a \emph{direct sum} of the finite
family $\{V_i\}_i$ of objects of $\cat$ if there exist retracts
$(V_i,\alpha_i,\beta_i)$ of $W$ with $\beta_i\alpha_j=0$ for $i\neq j$
and $\Id_W=\sum_i\alpha_i\beta_i$.
An object which is a direct sum of simple objects is called {\em
  semi-simple}.

\subsection{Colored ribbon graph invariants}
Let $\cat$ be a pivotal $\FK$-category.  
A ribbon graph is formed from several oriented framed edges colored by objects of $\cat$ and several coupons colored with morphisms of $\cat$. 
We say a $\cat$-colored ribbon graph in $\R^2$ (resp.\ $S^{2}=\R^2\cup\{\infty\}$) is called \emph{planar} (resp.\ \emph{spherical}).  
Let $\Ffun$ be the usual Reshetikhin-Turaev functor from the category of $\cat$-colored planar ribbon graphs to $\cat$ (see \cite{Tu}).    

Let
$T\subset S^2$ be a closed $\cat$-colored ribbon
graph.  Let $e$ be an edge of $T$ colored with a simple object $V$ of
$\cat$. Cutting $T$ at a point of $e$, we obtain a $\cat$-colored ribbon graph
$T_V $ in $\R\times [0,1]$ where
$\Ffun(T_V)\in\End(V)=\FK \Id_V$.   We call $T_V $ a \emph{cutting
  presentation} of $T$ and let $\ang{T_V} \, \in \FK$ denote the isotopy
invariant of $T_V$ defined from the equality $\Ffun(T_V )= \, \ang{ T_V}\,
\Id_V$.

Let $\A$ be a class of simple objects of $\cat$ and $\qd:\A\to\FK^\times$ be a
mapping such that $\qd(V)=\qd(V^*)$ and $\qd(V)=\qd(V')$ if $V$ is isomorphic
to $V'$.  We say  $(\A,\qd)$ is \emph{t-ambi pair} if for any closed $\cat$-colored trivalent ribbon graph $T$ with any two
cutting presentations $T_V$ and $T_{V'}$,  $V, V'\in\A$ the following equation holds:
\begin{equation*}
  \qd(V)\ang{T_V}=\qd(V')\ang{T_V'}.
\end{equation*}

\subsection{$\Gr$-graded and generically $\Gr$-semi-simple
  categories}\label{SS:GSpher}
Let $\Gr$ be a  group.    
A pivotal $\FK$-category is {\em $\Gr$-graded} if
  for each $g\in \Gr$ we have a non-empty full subcategory $\cat_g$ of
  $\cat$ such that
  \begin{enumerate}
  \item $\unit \in \cat_e$, (where $e$ is the identity element of $G$)
  \item  $\cat=\bigoplus_{g\in\Gr}\cat_g$, 
  \item  if $V\in\cat_g$,  then  $V^{*}\in\cat_{g^{-1}}$,
  \item  if $V\in\cat_g$, $V'\in\cat_{g'}$ then $V\otimes
    V'\in\cat_{gg'}$,
  \item  if $V\in\cat_g$, $V'\in\cat_{g'}$ and $\Hom_\cat(V,V')\neq 0$, then
    $g=g'$.  
  \end{enumerate}
For a subset $\X\subset\Gr$ we say:
\begin{enumerate}
\item $\X$ is \emph{symmetric} if $\X^{-1}=\X$,
\item $\X$ is \emph{small} in $\Gr$ if the group $\Gr$ can not be covered by a
  finite number of translated copies of $\X$, in other words, for any
  $ g_1,\ldots ,g_n\in \Gr$, we have $ \bigcup_{i=1}^n (g_i\X) \neq\Gr$.
\end{enumerate}

  \begin{enumerate}
  \item A $\FK$-category $\cat$ is \emph{semi-simple} if all its objects are semi-simple.
  \item A $\FK$-category $\cat$ is \emph{finitely semi-simple} if it is
    semi-simple and has finitely many isomorphism classes of simple
    objects.
  \item A $\Gr$-graded 
    category $\cat$ is a 
    \emph{generically finitely
    $\Gr$-semi-simple} category 
    if there exists a small
    symmetric subset $\X\subset \Gr$ such that for each
    $g\in\Gr\setminus\X$, $\cat_g$ is 
    finitely semi-simple. 
    By a
    \emph{generic simple} object we mean a  simple object of $\cat_g$ for some $g\in\Gr\setminus\X$.  
   \end{enumerate}

The notion of generically $\Gr$-semi-simple categories appears in
\cite{GPT2,GP3} through the following generalization of fusion
categories (in particular, fusion categories satisfy the following
definition when $\Gr$ is the trivial group, $\X=\emptyset$ and
$\qd=\bb=\qdim_\cat$ is the quantum dimension): 
\begin{definition}[Relative $\Gr$-spherical category]\label{D:G-spherical} 
Let $\cat$ be a generically finitely $\Gr$-semi-simple 
  pivotal 
 $\FK$-category with
  small
    symmetric subset  $\X\subset\Gr$ and let $\A$ be the class of generic
  simple objects of $\cat$. We say that $\cat$ is
  \emph{$(\X,\qd)$-relative $\Gr$-spherical} if
  \begin{enumerate}
  \item \label{ID:G-sph6} there exists a map $\qd:\A\rightarrow
    \FK^{\times}$ such that $(\A,\qd)$ is a t-ambi pair,
  \item \label{ID:G-sph7} there exists a map $\bb:\A\to \FK$ such that
    $\bb(V)=\bb(V^*)$, $\bb(V)=\bb(V')$ for any isomorphic objects $V, V'\in
    \A$ and for any $g_1,g_2,g_1g_2\in \Gr\setminus\X$ and $V\in \Gr_{g_1g_2}$
    we have
  \begin{equation*}\label{eq:bb}
    \bb(V)=\sum_{V_1\in irr(\cat_{g_1}),\, V_2\in irr(\cat_{g_2})}
    \bb({V_1})\bb({V_2})\dim_\FK(\Hom_\cat(V, V_1\otimes V_2))
  \end{equation*}
  where $irr(\cat_{g_i})$ denotes a representing set of isomorphism classes
  of simple  objects of $\cat_{g_i}$.
  \end{enumerate}
\end{definition}

  If $\cat$ is   a category with
   the above data, for brevity we say $\cat$ is a \emph{relative
    $\Gr$-spherical category}.
In \cite{GPT2}, to construct a 3-manifold invariant from a relative $\Gr$-spherical category $\cat$ the authors assume that $\cat$ has a technical requirement called \emph{basic data}.  The following lemma (proved in \cite{GP2}) says that in most situations $\X$ can be enlarged so that $\cat$ has basic data.   This lemma implies that we can assume the categories we considered in this paper have basic data.   

\begin{lemma}\label{L:basicdata}
  If no object of $\A$ is isomorphic to its dual, then $\cat$ contains a basic
  data.  In particular, basic data exists if $\X$ contains the set $\{g\in \Gr
  : g=g^{-1}\}$.
\end{lemma}



\subsection{Traces on ideals in pivotal categories}\label{SS:trace}
In this subsection we recall some facts about right traces in a pivotal $\FK$-category $\cat$, for more details see \cite{GPV, GKP2}. 
In this paper we will use right traces to show that a t-ambi pair exists.  
  
By a \emph{right ideal} of  $\cat$ we mean a full subcategory $\ideal$ of $\cat$ such that:  
\begin{enumerate}
\item  If $V\in \ideal$ and $W\in \cat$, then $V\otimes W \in \ideal$.
\item If $V\in \ideal$ and if $W\in \cat$ is a retract of $V$,
  then $W\in \ideal$.  
\end{enumerate}
 A \emph{right trace} on a right ideal  $\ideal$ is a family of linear functions
$$\{\qt_V:\End_\cat(V)\rightarrow \FK \}_{V\in \ideal}$$
such that:
\begin{enumerate}
\item  If $U,V\in \ideal$ then for any morphisms $f:V\rightarrow U $ and $g:U\rightarrow V$  in $\cat$ we have 
\begin{equation*}
\qt_V(g f)=\qt_U(f  g).
\end{equation*} 
\item \label{I:Prop2Trace} If $U\in \ideal$ and $W\in \cat$ then for any $f\in \End_\cat(U\otimes W)$ we have
\begin{equation*}
\qt_{U\otimes W}\bp{f}=\qt_U \bp{(\Id_U \otimes \tev_W)(f\otimes \Id_{W^*})(\Id_U \otimes \coev_W)}
\end{equation*}
\end{enumerate}


Next we recall how to construct a right trace.  Given an object $V$ of $\cat$ we define the ideal generated by $V$ as
 $$ \ideal_{V}=\{W \in  \cat\, | \, W \text{ is a retract of } V\otimes X  \text{ for some object } X \}.$$ 
In \cite{GPV} the notion of a  right ambidextrous simple object is developed (see Sections 4.2 and 4.4 of \cite{GPV}).  Theorem 10 of \cite{GPV} implies:
\begin{theorem}[\cite{GPV}]\label{T:rigthambiTrace}
If $V$ is a right ambidextrous simple object then there exists a non-zero right trace $\{\qt_V\}$ on the ideal $\ideal_{V}$; this trace is unique up to multiplication by a non-zero scalar.
\end{theorem}

Now we will recall a way to show a simple object is right ambidextrous.  
Let   $V$ be a simple object in $\cat$.   We fix a direct sum decomposition of $V\otimes V^*$ into indecomposable objects
$W_i$ indexed by a set $I$: 
\begin{equation}\label{E:Decomp}
V\otimes V^* = \bigoplus_{k\in I}W_k.
\end{equation}
Let  $i_{k}: W_{k}\to V\otimes V^{*}$ and $p_{k}: V \otimes V^{*} \to
W_{k}$ be the 
morphisms corresponding to this decomposition.  In particular,
$\sum_{k\in I}  i_{k}p_{k}=\Id_{V \otimes V^{*}}$ and $p_{k}i_{k} = \Id_{W_{k}}$, for all $k \in I$.
From Lemma 3.1.1 of \cite{GKP2} we have the following lemma:
\begin{lemma}[\cite{GKP2}]\label{L:Uniquejj}
There exists unique  $j,j'\in I$ so that 
  \begin{enumerate}
  \item   $\Hom_{\cat}(\unit , W_{j}) $ is non-zero and is spanned by $p_{j}
    \coev_{V}$  and 
    \item 
    $\Hom_{\cat}(W_{j'} , \unit)$ is non-zero and is spanned by
    $\tev_{V} i_{j'}$.
  \end{enumerate}
  \end{lemma}
Theorem 3.1.3. of \cite{GKP2} gives the following theorem.  
\begin{theorem}[\cite{GKP2}] \label{T:RightAmbij=j'}
The simple object $V$ is right ambidextrous if and only if $j=j'$.
\end{theorem}

Finally, let us explain how to produce a t-ambi pair from a right trace.   Let $\qt$ be a right trace on a right ideal $\ideal$ of $\cat$.  Let $\qd$ be the  modified dimension associated with $\qt$ defined by $\qd(V)=\qt_V(\Id_V)$ for $V\in \ideal$.  Set
 $$
  \B=\{V\in\ideal\cap\ideal^{*}\,|\, V\text{ is simple and
  }\qd(V)=\qd(V^{*})\}
  $$
  where $\ideal^*=\{V\in \cat \; | \; V^*\in \ideal\}$. 
  The following theorem is Corollary 7 in \cite{GPV}.
\begin{theorem}[\cite{GPV}]\label{T:Bt-ambi}
The pair $(\B,\qd)$ is a t-ambi pair.
\end{theorem}

\section{Quantum $\slto$ at roots of unity}

\subsection{Notation}
Fix a positive integer $\rt\geq 3$ and let $q=e^\frac{2\pi\sqrt{-1}}{\rt}$ be a
$\rt^{th}$-root of unity.  Set
$$\rt'= 
\left\{\begin{array}{cc}
 \rt & \text{ if $\rt$ is odd} \\
  \rt/2  & \text{ if $\rt$ is even} \end{array}\right. . $$
 We use two quotients of the complex numbers: $\C/\Z$ and $\C/\rt\Z$.  We will use greek letters to denote elements of $\C$.  We will denote elements in $\C/\Z$ and $\C/\rt\Z$ with bars and tildes, respectively.  
 In this paper, both $\C/\Z$ and $\C/\rt\Z$ are abelian groups induced from the addition in $\C$.

   For $\alpha\in \C$ let $\tilde{\alpha}$ be the element of $ \C/\rt\Z$ such that $\alpha$ is in the equivalence class of  $\tilde{\alpha}$.  In other words, $\alpha$ maps to $\tilde{\alpha}$ under the obvious map $\C\to \C/\rt\Z$.  Similarly, for $\alpha\in \C$ or $\tilde{\alpha}\in \C/\rt\Z$ let $\bar\alpha\in\C/\Z$ such that $\alpha$ or $\tilde{\alpha}$ is mapped to $ \bar{\alpha}$ under the map $\C\to \C/\Z$ or $\C/\rt\Z\to \C/\Z$, respectively.  
For $x$ in $ \C$ or $\C/\rt\Z$ set $\{x\}=q^x-q^{-x}$ and  $ [x]=\frac{\{x\}}{\{1\}}$. 

\subsection{Superspaces}
A \emph{superspace} is a $\Z_{2}$-graded vector space $V=V_{\p 0}\oplus V_{\p 1}$ over $\C$.  We denote the parity of a homogeneous element $x\in V$ by $\p x\in \Z_{2}$.  We say $x$ is even (odd) if $x\in V_{\p 0}$ (resp. $x\in V_{\p 1}$). 
If $V$ and $W$ are $\Z_{2}$-graded vector spaces then the space of linear maps $\Hom_\C(V,W)$ has a natural $\Z_{2}$-grading given by $f \in \Hom_C(V,W)_{\p j}$ if $f(V_{\p i}) \subset W_{\p i+\p j}$ for $\p i, \p j \in \Z_2$.
 Throughout, all modules of over a $\Z_{2}$-graded ring will be $\Z_{2}$-graded modules.

 \subsection{The superalgebra $\Uq$}\label{SS:UqDef} Let $A=(a_{ij})$ be the $2\times 2$ matrix given by $a_{11}=2$,
$a_{12}=a_{21}=-1$ and $a_{22}=0$. 
  Let $\Uq$ be the 
$\C$-superalgebra generated by the elements $K_{i}, K_i^{-1},E_{i}$ and $ F_{i}, $
  $i = 1,2$, subject to the relations: 
  \begin{align*}
 K_i^{\pm1} K_j &=  K_j K_i^{\pm1} , &  K_i^{-1} K_j^{\pm1} &=  K_j^{\pm1} K_i^{-1} , & 
  K_iF_jK_i^{-1}&=q^{-a_{ij}}F_j,
\end{align*}
\begin{align*}
K_iE_jK_i^{-1}&=q^{a_{ij}}E_j, & [E_{i},F_{j}]=&\delta_{i,j}\frac{K_i-K_i^{-1}}{q-q^{-1}}, 
 \end{align*}
\begin{align}
\label{E:QserreA2}
  E_{1}^{2}E_{2}-(q+q^{-1})E_{1}E_{2}E_{1}+E_{2}E_{1}^{2}&=0, &
  F_{1}^{2}F_{2}-(q+q^{-1})F_{1}F_{2}F_{1}+F_{2}F_{1}^{2}&=0
\end{align}
where $[,]$ is the super-commutator given by $[x,y]=xy-(-1)^{\p x \p
 y}yx$.  All generators are even except for $E_{2}$ and $F_{2}$ which
 are odd. 
The algebra $\Uq$ is a Hopf algebra where the coproduct, counit and
antipode are defined by
\begin{align*}
\label{}
     \Delta ({E_{i}})= & E_{i}\otimes 1+ K_i^{-1} \otimes E_{i}, &
   \epsilon(E_{i})= & 0 & S(E_{i})=&-K_iE_{i}\\
%
     \Delta ({F_{i}}) = & F_{i}\otimes K_i+ 1 \otimes F_{i}, &
   \epsilon(F_{i})= &0 & S(F_{i})=&-F_{i} K_i^{-1}\\
%
      \Delta(K_i)=&K_i\otimes K_i
  &\varepsilon(K_i)=&1,
  & S(K_i)=&K_i^{-1},
    \\
  \Delta(K_i^{-1})=&K_i^{-1}\otimes K_i^{-1}
  &\varepsilon(K_i^{-1})=&1,
  & S(K_i^{-1})=&K_i.
 \end{align*}

\subsection{ Representations of $\Uq$}
 A $\Uq$-module $V$ is a \emph{weight module} if $V$ is finite dimensional and both $K_1$ and $K_2$ act diagonally on $V$. 
  Let $\catd$ be the tensor category of finite dimensional $\Z_2$-graded  weight $\Uq$-modules, whose unit object is the trivial module $\unit=\C$.  It is easy to see that this category is a $\C$-category.



A direct calculation shows that $S^2(x)=K_2^{-2}xK_2^{2}$ for all $x\in \Uq$.
  Thus, the square of the antipode is equal the conjugation of a
group-like element and so $\catd$ is a pivotal category (see \cite[Proposition 2.9]{Bi}).  In particular, for any
object $V$ in $\catd$, the dual object and the duality morphisms are defined as
follows: $V^* =\Hom_\C(V,\C)$ and
\begin{align}\label{E:DualityForCat}
\coev_V :\, & \C \rightarrow V\otimes V^{*} \text{ is given by } 1 \mapsto
  \sum v_j\otimes v_j^*,\notag
  \\
 \ev_{V}:\, & V^*\otimes V\rightarrow \C \text{ is given by } f\otimes w
  \mapsto f(w),\notag
  \\
\tcoev_V:\, & \C \rightarrow V^*\otimes V \text{ is given by } 1 \mapsto \sum
  (-1)^{\p{v_j}} v_j^*\otimes K_{2}^{2}v_j,
  \\
\tev_{V}:\, & V\otimes V^{*}\rightarrow \C \text{ is given by } v\otimes f
  \mapsto (-1)^{\p f\p{v}} f(K_{2}^{-2}v),\notag
\end{align}
where $\{v_j\}$ is a basis of $V$ and $\{v_j^*\}$ is the dual basis
of $V^*$.  These morphisms satisfy the compatibility conditions of a pivotal category.  

The simple $\Uq$-modules have been studied in \cite {AAB}.  Here we will consider what they call the \emph{typical type A representations}:  let $\omega\in \{\pm1\}$, $0\leq n\leq \rt'-1$ and $\tilde{\alpha}\in \C/\rt\Z$ then there exists a highest weight module $V(\omega, n, \tilde{\alpha})$ with highest weight vector $v$ such that $$E_i v=0, \;\; K_1v= \omega q^n v\;\; \text{  and } \;\;K_2v= q^{\tilde{\alpha}}v.$$
In \cite{AAB} it is shown that under certain conditions $V(\omega, n, \tilde{\alpha})$ is a simple module of dimension $4(n+1)$.  
Let us now give these conditions.  
To simplify notation if $\omega=1$ we set $V(n, \tilde{\alpha})=V(1, n, \tilde{\alpha})$.  
\begin{proposition}[\cite{AAB}, page 873]
If $[\tilde{\alpha}]\cdot [\tilde{\alpha}+n+1]\neq0$  then $V(n,\tilde{\alpha})$ is simple.  
\end{proposition}
\begin{remark} We use slightly different notation than  \cite {AAB}.  Our module $V(n, \tilde{\alpha})$ corresponds to the module from \cite[page 873]{AAB}  with the following parameters: 
$$\omega =1, \ N=n+1, \ \lambda_1=q^n,\ \mu_1=n=N-1, \ \lambda_2=q^{\alpha}, \ \mu_2=\alpha.$$
\end{remark}

Since 
$$[x]=0\Leftrightarrow  \frac{q^x-q^{-x}}{q-q^{-1}}=0\Leftrightarrow  q^{x}-q^{-x}=0\Leftrightarrow  q^{2x}=1\Leftrightarrow  x\in  \frac {\rt}{2} \Z$$ 
then the above proposition implies that $V(n,\tilde{\alpha})$ is simple if $\tilde{\alpha}\in  (\frac{\rt}{2}\Z)/\rt\Z\cup ((\frac{\rt}{2}\Z)-(n+1))/\rt\Z $.  In particular, if $\bar{\alpha}\notin \{ \bar{0}, \bar{\frac{\rt}{2}} \} \subseteq \C/\Z$ then $V(n,\tilde{\alpha})$ is simple. \label{R:simple}

\begin{theorem}[\cite{AAB}] \label{T:Structure} 
Let $n\in  \{0,...,{\rt}'-1 \}$ and  $ \tilde{\alpha} \not\in  (\frac{\rt}{2}\Z)/\rt\Z\cup ((\frac{\rt}{2}\Z)-(n+1))/\rt\Z$.  Then $V(n,\tilde{\alpha })$ has a basis 
 $ \{ w_{\rho ,\sigma ,p}| p\in \{ 0,...,n \}; \rho ,\sigma \in \{ 0,1\} \}$ whose action of $\Uq$ is given by: 
\begin{align}\label{E:ActionVntildea_K}
K_1\cdot w_{\rho ,\sigma ,p}&= q^{\rho -\sigma +n-2p} w_{\rho ,\sigma ,p},  &
K_2\cdot w_{\rho ,\sigma ,p}&= q^{\tilde{\alpha }+\sigma+p} w_{\rho ,\sigma ,p},
\end{align}
\begin{align}\label{E:ActionVntildea_F}
F_1\cdot w_{\rho ,\sigma ,p}&= q^{\sigma-\rho} w_{\rho ,\sigma ,p+1}-\rho (1-\sigma )q^{-\rho }w_{\rho -1,\sigma +1,p}, &
F_2\cdot w_{\rho ,\sigma ,p}&= (1-\rho ) w_{\rho+1 ,\sigma ,p},
\end{align}
\begin{equation}\label{E:ActionVntildea_E1}
E_1\cdot w_{\rho ,\sigma ,p}= -\sigma (1-\rho) q^{n-2p+1} w_{\rho+1 ,\sigma-1 ,p}+ [p] [n-p+1] w_{\rho,\sigma ,p-1},
\end{equation}
\begin{equation}\label{E:ActionVntildea_E2}
E_2\cdot w_{\rho ,\sigma ,p}= \rho [\tilde{\alpha }+p+\sigma] w_{\rho-1 ,\sigma ,p}+\sigma (-1)^{\rho} q^{-\tilde{\alpha}-p} w_{\rho,\sigma-1,p+1}. 
\end{equation}
Here the super grading of this basis is given by $ \p{w_{\rho ,\sigma ,p}}={\rho+\sigma}\in \Z/2\Z$.
\end{theorem}
Since $w_{1,1,n}$ is a lowest weight vector of $V(n,\tilde{\alpha } )$ with weight $(-n, \tilde{\alpha} +\tilde{n}+\tilde{1} )$ then we have
$$V(n,\tilde{\alpha } )^{\ast }=V(n, -\tilde{\alpha }-\tilde{n}-\tilde{1}).$$ We will use the modules of the form $V(0,\tilde\alpha)$ extensively. With this in mind we highlight the structure of such modules.   If $ \tilde{\alpha} \not\in  (\frac{\rt}{2}\Z)/\rt\Z\cup ((\frac{\rt}{2}\Z)-(n+1))/\rt\Z$ then $V(0,\tilde{\alpha })$ is a 4-dimensional module with the following action:
\begin{center}
  \begin{tabular}{ | l || c | c | c | c | }
    \hline
      & $w_{0,0}$ & $w_{1,0}$ & $w_{0,1}$ & $w_{1,1}$  \\ 
      & $ \rho =0, \sigma =0$  & $ \rho =0, \sigma =0$ & $ \rho =0, \sigma =0$ & $ \rho =0, \sigma =0$ \\ \hline
      $K_1$ & $w_{0,0}$ & $qw_{1,0}$ & $q^{-1}w_{0,1}$ & $w_{1,1}$  \\ \hline
      $K_2$ & $q^{\tilde{\alpha } }w_{0,0}$ & $q^{\tilde{\alpha } }w_{1,0}$ & $q^{\tilde{\alpha }+1 }w_{0,1}$ & $q^{\tilde{\alpha }+1 }w_{1,1}$  \\ \hline
      $E_1$ & $0$ & $0$ & $w_{1,0}$ & $0$  \\ \hline
      $E_2$ & $0$ & $[\tilde{\alpha } ]w_{0,0}$ & $0$ & $[\tilde{\alpha }+1 ]w_{0,1}$  \\ \hline
      $F_1$ & $0$ & $w_{0,1}$ & $0$ & $0$  \\ \hline
      $F_2$ & $w_{1,0}$ & $0$ & $w_{1,1}$ & $0$  \\ \hline
  \end{tabular}
\end{center}

We will use the following lemma in the proof of the Decomposition Lemma below.

\begin{lemma}\label{L:FourDirSum} Let $V$ be an object in $\cat$. 
Suppose  $V_1,...,V_n$ are simple submodules of $V$ such that  $V_i$ is not isomorphic to $ V_j$ for all $i\not=j$ and 
$$\dim(V_1)+...+\dim(V_n)=\dim(V).$$
Then $V=V_1 \oplus... \oplus V_n.$
\end{lemma}
\begin{proof} 
Consider the statement $P(k)$: if $ i_1,...,i_k,j \in \{1,...,n\} $  are all different then $$V_j \cap(V_{i_1}+...+V_{i_k})=\{ 0 \}.$$
If $P(k)$ were true for $k \in \{1,...,n-1 \}$ then $ V_1\oplus...\oplus V_n$ would be a submodule of $V$.  The hypothesis on the dimensions would then imply 
 $$V=V_1\oplus...\oplus V_n.$$  Thus, it suffices to prove $P(k)$ holds for $k \in \{1,...,n-1 \}$.  We will do this by induction.

First, we will show $P(1)$ holds. Let $i,j \in \{1,...,n\}$ such that $V_i \cap V_j \neq \{0\}$ and $i\not=j$.   
Therefore, there is a non-zero vector in $V_i\cap V_j$ which generates a submodule  $W$ of 
$V_i\cap V_j$.  In particular, $W$ is a submodule of both $V_i$ and $V_j$.  Since these modules are simple we have $W$ is equal to both  $V_i$ and $V_j$.  Thus, $V_i=V_j$ which is a contradiction. 

Next assuming $P(k)$ is true we will show $P(k+1)$ holds.   
Let $i_1,...,i_{k+1}$ and $j$ be unique elements of $\{1,...,n\}$.  
Suppose by contradiction that $V_j \cap(V_{i_1}+...+V_{i_{k+1}})\neq\{ 0 \}$ and let $v$ be a non-zero vector in this intersection. Let  $<v>$ be the non-zero module generated by $v$.  
Since $v \in V_j$ then $<v>$  is a submodule of $V_j$.   But $V_j$ is simple so $<v>=V_j$.  
Similarly, $v\in V_{i_1}+...+V_{i_{k+1}}$ implying $<v>$ is in this sum and we conclude 
$$V_j\subseteq V_{i_1}+...+V_{i_{k+1}}.$$

From the induction step for $P(k)$, we deduce that $V_{i_1}+...+V_{i_{k+1}}=V_{i_1}\oplus...\oplus V_{i_{k+1}}$.  
 Combining the last two observations, we have $V_j$ is a submodule of  $V_{i_1}\oplus...\oplus V_{i_{k+1}}.$  
This implies $\Hom_\cat(V_j,V_{i_1}\oplus...\oplus V_{i_{k+1}}) \neq \{0\}$ since the inclusion morphism is in this space.

On the other hand, since the simple modules $V_j$ and $V_{i_s}$ are non-isomorphic for $s \in \{1,..,k+1 \}$ then $\Hom_\cat(V_j,V_{i_s})= \{0 \}$.  This implies  $$\Hom_\cat(V_j,V_{i_1}\oplus...\oplus V_{i_{k+1}})=\Hom_\cat(V_j,V_{i_1})\oplus...\oplus \Hom_\cat(V_j, V_{i_{k+1}})=\{0\}.$$
But above we showed this homomorphism space was non-zero so we have a contradiction.   Thus, the induction step is complete.
  \end{proof}

\begin{lemma}(Decomposition Lemma)\label{L:DecompV0Vn}
Let $\tilde{\alpha}, \tilde{\beta}\in \C/{\rt}\Z$ with $\bar{\alpha },\bar{\beta }\not\in \{ \bar{0}, \bar{\frac\rt2} \}$ and $\bar{\alpha }+\bar{\beta }\not \in \{ \bar{0} ,\bar{\frac\rt2} \}$.  Then for any $n\in  \{ 0,...,{\rt}'-2 \} $ we have 
$$V(0,\tilde{\alpha } )\otimes V(n,\tilde{\beta } )=V(n,\tilde{\alpha }+\tilde{\beta})\oplus  V(n+1,\tilde{\alpha }+\tilde{\beta})\oplus  V(n-1,\tilde{\alpha }+\tilde{\beta}+1)\oplus V(n,\tilde{\alpha }+\tilde{\beta}+1)$$
where we use the convention that $V(-1,\tilde{\alpha }+\tilde{\beta}+1) =0$ in the case when $n=0$.
\end{lemma}
\begin{proof} 
We will prove the case when $n\not=0$ (the case n=0 will be analogous).
Since $\bar{\alpha },\bar{\beta }\notin \{\bar 0, \bar{\frac\rt2}\}$, it means that $V(0,\tilde{\alpha } )$ and $V(n,\tilde{\beta} )$ have the structure described in 
Theorem \ref{T:Structure}: let  $\{w_{\rho,\sigma}^{({0,\tilde{\alpha}})}\}$ and $
\{w_{\rho,\sigma,p}^{({n,\tilde{\beta}})}\}$ be the bases of $V(0,\tilde{\alpha } )$ and $V(n,\tilde{\beta} )$, respectively. 

  We will prove that there are four highest weight vectors:
$$ v_{(n,\tilde{\alpha }+\tilde{\beta})}, v_{(n+1,\tilde{\alpha }+\tilde{\beta})}, v_{(n-1,\tilde{\alpha }+\tilde{\beta}+1)}, v_{(n,\tilde{\alpha }+\tilde{\beta}+1)} \in V(0,\tilde{\alpha } )\otimes V(n,\tilde{\beta } )$$
where the weight of $v_{(i,\tilde{\gamma }  )}$ is $(q^i,q^{\tilde{\gamma }} )$. 
First, clearly $v_{(n,\tilde{\alpha }+\tilde{\beta})}:=w_{0,0}^{({0,\tilde{\alpha}})}\otimes w_{0,0,0}^{({n,\tilde{\beta}})}$ is a highest weight vector of weight $(n,\tilde{\alpha }+\tilde{\beta})$.   
Second, we want to find a highest weight vector $v_2=v_{(n+1,\tilde{\alpha }+\tilde{\beta})}$ with weight $(q^{n+1},q^{\tilde{\alpha}+\tilde{\beta}})$.  
We'll search for $v_2$ as a combination of the form:
$$v_2=w_{0,0}^{({0,\tilde{\alpha}})}\otimes w_{1,0,0}^{({n,\tilde{\beta}})}+c\cdot w_{1,0}^{({0,\tilde{\alpha}})}\otimes w_{0,0,0}^{({n,\tilde{\beta}})}.$$ 
To find $c$ we check that $E_1$ and $E_2$ act by zero.  For any $c$ we have   
$E_1(v_2)=0$.  On the other hand, 
$E_2(v_2)=0$ implies $c=-q^{-\tilde{\alpha}}\cdot  \frac {[\tilde{\beta}]}{[\tilde{\alpha}]}$.  
So, 
$$v_2=w_{0,0}^{({0,\tilde{\alpha}})}\otimes w_{1,0,0}^{({n,\tilde{\beta}})}-q^{-\tilde{\alpha}}\cdot  \frac {[\tilde{\beta}]}{[\tilde{\alpha}]}\cdot w_{1,0}^{({0,\tilde{\alpha}})}\otimes w_{0,0,0}^{({n,\tilde{\beta}})}$$
is a highest weight vector.  

Third, we want a highest weight vector 
$v_3=v_{(n-1,\tilde{\alpha }+\tilde{\beta}+1)}$
of the form
$$c_1 w_{0,0}^{({0,\tilde{\alpha}})}\otimes w_{0,1,0}^{({n,\tilde{\beta}})}+c_2w_{0,0}^{({0,\tilde{\alpha}})}\otimes w_{1,0,1}^{({n,\tilde{\beta}})}+c_3w_{1,0}^{({0,\tilde{\alpha}})}\otimes w_{0,0,1}^{({n,\tilde{\beta}})}+c_4 w_{0,1}^{({0,\tilde{\alpha}})}\otimes w_{0,0,0}^{({n,\tilde{\beta}})}.$$
After checking the conditions which come from the action of $E_1$ and $E_2$, and setting $c_2=1$ we obtain:
\begin{multline*}
v_3=q^{-(n+1)}[1][n]\cdot w_{0,0}^{({0,\tilde{\alpha}})}\otimes w_{0,1,0}^{({n,\tilde{\beta}})}+w_{0,0}^{({0,\tilde{\alpha}})}\otimes w_{1,0,1}^{({n,\tilde{\beta}})}
\\
- \frac{1}{[\tilde{\alpha}]} 
(q^{-(\tilde{\alpha}+\tilde{\beta}+n+1)}[1][n]
+q^{- \tilde{\alpha}}[\tilde{\beta}+1]) \cdot w_{1,0}^{({0,\tilde{\alpha}})}\otimes w_{0,0,1}^{({n,\tilde{\beta}})}
\\
- \frac{q^{-1}[1][n]}{[\tilde{\alpha}]} (q^{-(\tilde{\alpha}+\tilde{\beta}+n+1)})[1][n]+q^{-\tilde{\alpha}}[\tilde{\beta}+1])\cdot w_{0,1}^{({0,\tilde{\alpha}})}\otimes w_{0,0,0}^{({n,\tilde{\beta}})}.
 \end{multline*}
Similarly we obtain:
\begin{multline*}
v_4=v_{(n,\tilde{\alpha }+\tilde{\beta}+1)}=-q^{-\tilde{\alpha}}\frac{[\tilde{\alpha}]}{[\tilde{\beta}+1]}w_{0,0}^{({0,\tilde{\alpha}})}\otimes w_{1,1,0}^{({n,\tilde{\beta}})}
\\
+q^{-\tilde{\alpha}-1}\frac{[\tilde{\beta}]}{[\tilde{\alpha}+1]}(q^{-n}+q^{-1-\tilde{\beta}}\frac{[1][n]}{[\tilde{\beta}+1]})w_{1,1}^{({0,\tilde{\alpha}})}\otimes w_{0,0,0}^{({n,\tilde{\beta}})}
\\
+(q^{-n}+q^{-1-\tilde{\beta}}\frac{[1][n]}{[\tilde{\beta}+1]})w_{0,1}^{({0,\tilde{\alpha}})}\otimes w_{1,0,0}^{({n,\tilde{\beta}})}
\\
+w_{1,0}^{({0,\tilde{\alpha}})}\otimes w_{0,1,0}^{({n,\tilde{\beta}})}-\frac{q^{-\tilde{\beta}}}{[\tilde{\beta}+1]}w_{1,0}^{({0,\tilde{\alpha}})}\otimes w_{1,0,1}^{({n,\tilde{\beta}})}.
\end{multline*}

Consider the submodule $W_{(i,\tilde{\gamma }  )}$ of $ V(0,\tilde{\alpha } )\otimes V(n,\tilde{\beta } )$ generated by one of  the highest weight vectors $v_{(i,\tilde{\gamma }  )}$ constructed above.      
As mentioned above the classification of $U_q(sl(2|1))$-highest weight modules is given in \cite{AAB}.    
From this classification, since $W_{(i,\tilde{\gamma }  )}$ is a highest weight module of weight $(q^i,q^{\tilde{\gamma}})$, with $\bar{\gamma}=\bar{\alpha}+\bar{\beta}\notin \{\bar0, \bar{\frac{\rt}{2}}\}$ it follows that  $W_{(i,\tilde{\gamma }  )}$ is isomorphic to $V(i, \tilde{\gamma})$ 
and is a simple of dimension $4(i+1)$.  

Thus, we have 
$$\dim(W_{(n,\tilde{\alpha }+\tilde{\beta})})+\dim(W_{(n+1,\tilde{\alpha }+\tilde{\beta})})+\dim(W_{(n-1,\tilde{\alpha }+\tilde{\beta}+1)})+\dim(W_{(n,\tilde{\alpha }+\tilde{\beta}+1)})$$
$$=4((n+1)+(n+2)+n+(n+1))=4(4n+4)=16(n+1).$$
But $\dim(V(0,\tilde{\alpha } )\otimes V(n,\tilde{\beta } ))=4\cdot 4(n+1)=16(n+1)$.
So, the four submodules satisfy the conditions of Lemma \ref{L:FourDirSum}, which means that their direct sum is isomorphic to $V(0,\tilde{\alpha } )\otimes V(n,\tilde{\beta } )$.
\end{proof}
From the previous result, we obtain that, with some weight restrictions, the decomposition of the tensor product of two typical modules depends just on the total weight-sum, and it is independent on the two separate components. More precisely:  

\begin{corollary}\label{CorDecomp}
Consider $\tilde{\alpha}, \tilde{\beta}\in \C/{\rt}\Z, n\in  \{ 0,...,{\rt}'-1 \}$ such that 
$\bar{\alpha },\bar{\beta },\bar{\alpha }+\bar{\beta }\notin\{\bar0, \bar{\frac\rt2}\}$.  Then
$$ V(0,\tilde{\alpha})\otimes V(n,\tilde{\beta})\simeq V(0,\widetilde{\alpha+\epsilon})\otimes V(n,\widetilde{\beta-\epsilon})$$
for any $\epsilon \in \C$ such that $\overline{\alpha+\epsilon},\overline{\beta-\epsilon}\notin\{\bar0, \bar{\frac\rt2}\}$.
\end{corollary}

For $g\in \C/\Z$, let $\catd_g$ be the full subcategory of $\cat$ whose objects are all $\Uq$-module $V$ such that the central element  $K_2^{\rt}$ acts as multiplication by $q^{\rt g}$.  
In particular, for $0\leq n\leq \rt'-1$ and $\tilde{\alpha}\in \C/\rt\Z$ we have $V(n,\tilde{\alpha})\in \catd_{\bar{\alpha}}$.    This gives a $ \C/\Z$-grading on the category $\catd$ and we write  $\catd=\bigoplus_{g\in \C/\Z}\catd_g$. \label{E: Grading}

\subsection{The  subcategory $\cat$ of $\catd$}
Now we want to construct a subcategory $\cat$ of $\catd$ that will eventually (after taking a quotient) lead to our invariants for 3-manifolds. 

\begin{definition}\label{D:cat}
Set $\XX=(\frac{1}{4}\Z)/\Z$.  
Let $\cat$ the full sub-category of $\catd$ containing the trivial module and all retracts of a module of the form 
\begin{equation}\label{E:V0aV0aV0a}
V(0, \tilde{\alpha}_1)\otimes V(0, \tilde{\alpha}_2)\otimes...\otimes V(0, \tilde{\alpha}_n)
\end{equation}
  where $\tilde{\alpha}_1,...,\tilde{\alpha}_n\in \C/  \rt\Z$ such that $\bar{\alpha}_1,..., \bar{\alpha}_n\in (\C/\Z)\setminus \XX $. 
\end{definition} 
\begin{lemma}\label{L:catgraded}
The category $\cat$ is a 
$\C/\Z$-graded pivotal $\C$-category, where the grading and pivotal structure are   induced from $\catd$.  
\end{lemma}

\begin{proof}
Let $W_1$ and $W_2$ in $\cat$.  From the definition, for $j=1,2$, there exists $ \tilde{\alpha}_{j,1},...,\tilde{\alpha}_{j,{n_j}}\in \C/\rt\Z$, with 
$ \bar{\alpha}_{j,1},...,\bar{\alpha}_{j,{n_j}} \in (\C/\Z)\setminus \XX$ such that $W_j$ be is retract of\\ 
$V_j:=V(0,\tilde{\alpha}_{j,1} )\otimes ...\otimes V(0,\tilde{\alpha}_{j,{n_j}})$.  Let   $p_j:V_j\to W_j$ and $q_j: W_j \to V_j$ be the morphisms of this retract.  Then $V_1\otimes V_2$ is of the form of the module in Equation~\eqref{E:V0aV0aV0a} with all $\bar{\alpha}_{j,n}\notin \XX$.  It follows that $W_1\otimes W_2$ is an object of $\cat$ since it is a retract of $V_1\otimes V_2$ with maps $p_1\otimes p_2$ and $q_1\otimes q_2$.  Therefore, $\cat$ is a tensor category.
Moreover, $\cat$ is $\C$-category since it is a full sub-category of the $\C$-category $\catd$.  

Finally, we will check that $\cat$ is closed under duality.  
Let  $W\in \cat $.  Then $W$ is a retract of some $V:=V(0,\tilde{\alpha}_1)\otimes ...\otimes V(0,\tilde{\alpha}_n)$ such that $ \bar{\alpha}_1,...,\bar{\alpha}_n\in (\C/\Z)\setminus \XX$.  Then $W^\ast $ is a retract of $V^\ast $ and we have that:
$$V^*\cong \left(V(0, \tilde{\alpha}_{1})\otimes ...\otimes  V(0, \tilde{\alpha}_{n})\right)^*
\cong V(0, \tilde{\alpha}_{n})^*\otimes ...\otimes  V(0, \tilde{\alpha}_{1})^*$$
$$\cong  V(0, -\tilde{\alpha}_{n}-\tilde1) \otimes ...\otimes V(0, -\tilde{\alpha}_{1}-\tilde1).$$
But $-\bar{\alpha}_{n}-\bar1,...,-\bar{\alpha}_{1}-\bar1\in (\C/\Z)\setminus \XX$ so we have $W^*\in \cat$.  Thus, since $\cat$ is a full subcategory of $\catd$ then the duality morphisms of $\catd$ give a pivotal structure in $\cat$.  Finally, the $\C/\Z$-grading on $\catd$ induces a  $\C/\Z$-grading on $\cat$.  
\end{proof}

 The Decomposition Lemma  \ref{L:DecompV0Vn} says we can decompose the tensor product $V(0,\tilde{\alpha})\otimes V(0,\tilde{\beta})$ into simple modules if $\bar{\alpha}+\bar{\beta}\notin \{\bar0, \bar{\frac{\rt}{2}}\}$.  Given a module as in Equation \eqref{E:V0aV0aV0a}, the following lemma says we can always find a pair $\tilde{\alpha}_i,\tilde{\alpha}_j$ with this property.  This fact is one of the motivations for the choice of the set $\XX$.  

\begin{lemma}\label{L:i,j}
For any $\tilde{\alpha}_1,...,\tilde{\alpha}_n\in \C/\Z$ such that $\bar{\alpha}_1,...,\bar{\alpha}_n \in(\C/ \Z)\setminus \XX$ and
$$\bar{\alpha}_1+...+\bar{\alpha}_n \notin \{ \bar{0}, \bar{\frac{\rt}{2}} \} $$  
then there exist $i,j\in \{1,...,n\}$ such that $i\not=j$ and $\bar{\alpha}_i+\bar{\alpha}_j\notin \{ \bar{0}, \bar{\frac{\rt}{2}} \}.$
\end{lemma}

\begin{proof}
If $n=2$, we have just two numbers and from the hypothesis they have the desired sum.

Let us consider the case $n\geq3$ and let suppose by contradiction that 
\begin{equation}\label{E:alphaij} \overline{\alpha_i+\alpha}_j \in \{ \bar{0}, \bar{\frac{\rt}{2}} \},
\end{equation} 
for all $ i,j\in \{1,...,n\} $ with $i\not=j$.  
Up to a reordering, we can suppose that there exists $m \in \{ 2,...,n \}$ such that:
\begin{itemize}
\item  $\overline{\alpha_1+\alpha_i} = \bar{0}, \forall i \in \{2,...,m\}$ 
\item  $\overline{\alpha_1+\alpha_j} = \bar{\frac{\rt}{2}}, \forall j \in \{m+1,...,n\}$. 
\end{itemize}
This implies the following:
\begin{itemize}
\item  $\bar{\alpha}_i = - \bar{\alpha}_1, \forall i \in \{2,...,m\}$ 
\item  $\bar{\alpha}_i = \bar{\frac{\rt}{2}} - \bar{\alpha}_1 , \forall j \in \{m+1,...,n\}$. 
\end{itemize}
Now we have three cases.  

\textbf{Case 1.} If $m \geq 3$, then 
$\bar{\alpha}_2 = \bar{\alpha}_3= - \bar{\alpha}_1 $ which implies 
$$\bar{\alpha}_2 + \bar{\alpha}_3= -2 \bar{\alpha}_1 \notin \{ \bar{0}, \bar{\frac{\rt}{2}} \},  \ \ \  \text{since}  \ \bar{\alpha}_1 \notin \{  \bar{\frac{1}{2}}, \bar{\frac{\rt}{4}} \}$$
which is a contradiction with our supposition.

\textbf{Case 2.} If $n-m \geq 2$, then
$\bar{\alpha}_{m+1} = \bar{\alpha}_{m+2}= \bar{\frac{\rt}{2}} - \bar{\alpha}_1$
which implies 
$$\bar{\alpha}_{m+1}+\bar{\alpha}_{m+2}= -2 \bar{\alpha}_1.$$ Here as above this leads to a contradiction.  

\textbf{Case 3.} If we are not in the first two cases and $n\not=2$ then it means  $n =3$ and $m=2$.  In this case we have
\begin{itemize}
\item  $\bar{\alpha}_2 = - \bar{\alpha}_1$ 
\item  $\bar{\alpha}_3 = \bar{\frac{\rt}{2}} - \bar{\alpha}_1$. 
\end{itemize}  
The relations above lead to:
$$\bar{\alpha}_2+\bar{\alpha}_3= \bar{\frac{\rt}{2}}-2 \bar{\alpha}_1.$$ 
But from the initial supposition, we have that $\bar{\alpha}_2+\bar{\alpha}_3 \in \{ \bar{0}, \bar{\frac{\rt}{2}} \} $.

If $\bar{\alpha}_2+\bar{\alpha}_3=\bar{0} $, it implies that $\bar{\frac{\rt}{2}}-2 \bar{\alpha}_1=\bar{0}$, so $\bar{\alpha}_1=\bar{\frac{\rt}{4}} $ which contradicts that $\bar{\alpha}_1 \notin \XX=\frac14\Z/\Z$.

If $\bar{\alpha}_2+\bar{\alpha}_3= \bar{\frac{\rt}{2}}$, then $\bar{\frac{\rt}{2}}-2 \bar{\alpha}_1=\bar{\frac{\rt}{2}}$, and it means $\bar{\alpha}_1\in \{\bar{0},\bar{\frac12} $ which is impossible since $\bar{\alpha}_1 \notin \XX$.

Thus all cases lead to contradictions and so the lemma follows.
\end{proof}

The next part is devoted to an argument that will lead to the fact that the tensor product of simple modules in the alcove is commutative.  The proof uses the braiding of the ``un-rolled'' quantum $\UqH$, studied by Ha in \cite{Ha}.   In his paper he works with odd ordered roots of unity but as we observe his proof also works for even roots of unity (at least for the existence of a braiding, it may not extend to the twist).  

Let $\UqHa=\UqH$ be the superalgebra generated by the elements $K_{i}, K_i^{-1}, h_i, E_{i}$ and $ F_{i}, $
  $i = 1,2$, subject to the relations in  \eqref{E:QserreA2} and 
  $$
  [h_i,E_j] = a_{ij}E_j, \;\; [h_i,F_j]=-a_{ij}F_j, \;\; [h_i,h_j]=0,\;\; [h_i,K_j] = 0
  $$ 
  for  $i,j = 1,2.$   All generators are even except $E_2$ and $F_2$ which are odd.    This superalgebra is a Hopf algebra where the coproduct, counit and anitpode of $K_{i}, K_i^{-1}, E_{i}$ and $ F_{i} $ are given in Subsection \ref{SS:UqDef} and 
  $$
   \Delta(h_i)=h_i\otimes 1 +1\otimes h_i, \;\;   \epsilon(h_{i})=0,\;\; S(h_i)=-h_i
  $$
  for $i,j = 1,2.$
 
 For a $\UqHa$-module $V$ let  $q^{h_i}:V\to V$ be the operator defined by $q^{h_i}(v)=q^{\lambda_i}v$ where $v$ is a weight vector with respect to $h_i$ of weight $\lambda_i$.  The superalgebra ideal $I$ generated by $E_1^{\rt'}$ and $F_1^{\rt'}$ is a Hopf algebra ideal (i.e. an ideal in the kernel of the counit, a coalgebra coideal and stable under the antipode).  Let $\catdH$ be the category of finite dimensional $\UqHa/I$-modules with even morphisms such that $q^{h_i}=K_i$ as operators for $i=1,2$.    Since $\UqHa/I$ is a Hopf superalgebra then $\catdH$ is a tensor category.    
 Moreover, the maps given in Equation \eqref{E:DualityForCat}  define a pivotal structure on $\catdH$.  There is a forgetful functor from $\catdH$ to $\catd$ which forgets the action of $h_1$ and $h_2$.
Given two objects $V,W$ of $\catdH$ let $\mathcal{K}:V\otimes W\to V\otimes W$ be the operator defined by
$$\mathcal{K}(v\otimes w)=
q^{-\lambda_1\mu_2- \lambda_2\mu_1-2 \lambda_2\mu_2}v\otimes w
$$
where $h_iv=\lambda_iv$ and $h_iw=\mu_iw$ for $i=1,2$.    Consider the  truncated $R$-matrix:
\begin{equation}\label{E:TruR}
\check{R}=\sum_{k=0}^{\rt'-1} \frac{ \{1\}^k}{(k)_q!} E^k_1\otimes  F^k_1 \sum_{s=0}^{1} \frac{ (- \{ 1 \})^s}{(s)_q!} E_3^s \otimes F_3^s \sum_{t=0}^{1}\frac{(-\{ 1 \})^t}{(t)_q!} E^t_2 \otimes F^t_2 
\end{equation}
where $E_3=E_1E_2-q^{-1}E_2E_1, F_3=F_2F_1-qF_1F_2$, $(n)_q=\frac{1-q^n}{1-q}$ and $(n)_q!=(1)_q\cdot ...\cdot (n)_q$.
\begin{theorem}\label{T:braiding}
The family $\{c_{V,W}:V\otimes W\to W\otimes V\}_{V,W\in \catdH}$ defined by 
$$c_{V,W}(v\otimes w)=\tau(\check{R}\mathcal{K}(v\otimes w))$$ is a braiding on $\catdH$ where $\tau$ is the super flip map $\tau(v\otimes w)=(-1)^{\p w \p v}w\otimes v$.  
\end{theorem}
\begin{proof}
The proof is essentially given by Ha in \cite{Ha}.  As mentioned above, in Theorem~3.6 of \cite{Ha}, Ha proves the theorem for odd ordered roots of unity.  Ha's proof works for even ordered roots of unity as well.  In particular, before Proposition 3.5 of \cite{Ha} Ha uses the PBW basis of $\Uq$ to define an algebra $\mathcal{U}^{<}$.  In our case, when defining this algebra one should take powers of $E_1$ and $F_1$ from 0 to $\rt'-1$ not $\rt-1$.  Then use this algebra to define the projection $p:\Uq\to \mathcal{U}^{<}$ and the element
$$
\mathcal{R}^{<}=p\otimes p(\mathcal{R}_q)
$$  where $ \mathcal{R}_q$ is the $R$-matrix defined in \cite{KTol,Yam94} for $\Uq$ when $q$ is generic.  With these modifications the proofs of Proposition 3.5 and Theorem 3.6 in \cite{Ha} holds word for word for both the even and odd case.  Note that at the end of the proof of Theorem 3.6 in \cite{Ha} Ha says, ``The element $\mathcal{R}^{<}$  has no pole when $q$ is a root of unity of the order $\rt$.''  This is true in our case because we defined $\mathcal{R}^{<}$ using $p$ which only allows  powers of $E_1$ or $F_1$ smaller than $\rt'-1$ which is analogous to the definition of  $\check{R}$ above.  
\end{proof}

For $(n,\alpha)\in\N\times \C$ with $0\leq n\leq \rt'-1$ and $\bar \alpha\notin \{\bar0, \bar{\frac{\rt}{2}}\}$,  one can check directly that there is a $\UqH$-module $V^H(n,\alpha)$ with basis $\{w^{\alpha}_{\rho, \sigma ,p}| p\in \{ 0,...,n \}; \rho ,\sigma \in \{ 0,1\}\}$ whose action is given by   
\begin{align}\label{E:ActionVntildea_h}
h_1\cdot w^{\alpha}_{\rho ,\sigma ,p}&= (\rho -\sigma +n-2p) w^{\alpha}_{\rho ,\sigma ,p},  &
h_2\cdot w^{\alpha}_{\rho ,\sigma ,p}&= (\alpha +\sigma+p) w^{\alpha}_{\rho ,\sigma ,p},
\end{align}
and Equations \eqref{E:ActionVntildea_K}, \eqref{E:ActionVntildea_F}, \eqref{E:ActionVntildea_E1}, and \eqref{E:ActionVntildea_E2} with $\tilde{\alpha}$ replaced with $\alpha$. 
Moreover, by definition the operators $q^{h_i}=K_i$ are equal on $V^H(n,\alpha)$.

\begin{lemma}\label{L:El}
For $n\in \{ 0,...,\rt'-1 \}$ and $\alpha\in \C$ with $\bar \alpha\notin \{\bar0, \bar{\frac{\rt}{2}}\}$, then  
 the actions of $E^{\rt'}_1$ and $F^{\rt'}_1$ are zero on $ V(n,\tilde{\alpha})$  and $V^H(n,\alpha)$.
\end{lemma} 

\begin{proof}
We will prove the theorem for $V(n,\tilde{\alpha})$ the proof for $V^H(n,\alpha)$ is essential identical.  
Let us prove the action of $F^{\rt'}_1$ is zero on $V(n,\tilde{\alpha})$ the proof that $E^{\rt'}_1$ act as zero is similar and left to the reader.  

It is enough to prove that $F^{\rt'}_1w_{\rho, \sigma, p}=0$ where $w_{\rho, \sigma, p}$ is any of the basis vectors given in Theorem \ref{T:Structure}.   Equation \eqref{E:ActionVntildea_F} gives the action of $F_1$ on  $V(n,\tilde{\alpha})$.  In particular, if $\rho\neq 1$  and $\sigma \neq 0$ then $F_1w_{\rho, \sigma, p}= q^{\sigma-\rho} w_{\rho ,\sigma ,p+1}$.  Therefore, in this case, 
$$F_1^{\rt'}w_{\rho, \sigma, p}= q^{\rt'(\sigma-\rho)} w_{\rho ,\sigma ,p+\rt'}=0$$
since $w_{\rho ,\sigma ,i}=0$ if $i\geq \rt'$.  

Now a direct calculation implies:
$$
F_1^k w_{1 ,0 ,p} =q^{-k} w_{1 ,0 ,p+k} - q^{k-2}\left(\sum_{i=0}^{k-1}q^{-2i}\right)w_{0,1,p+k-1}.
$$
When $k=\rt'$ we see each of these terms is zero, since $w_{\rho ,\sigma ,i}=0$ if $i\geq \rt'$ and $\sum_{i=0}^{\rt'-1}q^{-2i}=\frac{1-q^{-2\rt'}}{1-q^{-2}}=0$.

\end{proof}

\begin{corollary}\label{C:VincatdH}
For $n\in \{ 0,...,\rt'-1 \}$ and $\alpha\in \C$ with  $\bar \alpha\notin \{\bar0, \bar{\frac{\rt}{2}}\}$ 
then  the  $\UqH$-module  $V^H(n,\alpha)$ is an object in $\catdH$.
\end{corollary}


\begin{lemma}(Commutativity Lemma)\label{comm} 
Let  $n,n'\in\N$ such that $0\leq n,n'\leq \rt'-1$ and 
$\tilde{\alpha},\tilde{\alpha}'\in \C/\rt\Z$ such that 
$\bar \alpha, \bar{\alpha}'\notin \{\bar0, \bar{\frac{\rt}{2}}\}$.  
 Let $\{w_{\rho, \sigma ,p}\}$ and $\{w'_{\rho', \sigma' ,p'}\}$ be the basis given in Theorem~\ref{T:Structure} for $V(n,\tilde{\alpha})$ and $V(n',\tilde{\alpha}')$, respectively.  Choose $\alpha,\alpha'\in \C$ such that $[\alpha]=\tilde{\alpha}$ and $[\alpha']=\tilde{\alpha}'$ in $ \C/\rt\Z$.   Then there exists an isomorphism
 $$\psi_{\alpha,\alpha'}:V(n,\tilde{\alpha})\otimes V(n',\tilde{\alpha}')\to  V(n',\tilde{\alpha}')\otimes V(n,\tilde{\alpha})$$
such that 
$$\psi_{\alpha,\alpha'}(w_{0,0,0}\otimes w'_{\rho', \sigma' ,p'})=
q^{-n(\alpha'+\sigma'+p'
)-\alpha(\rho'-\sigma' +n'-2p')-2\alpha(\alpha'+\sigma'+p')} w'_{\rho', \sigma' ,p'}\otimes w_{0,0,0}$$
and 
\begin{equation}\label{E:psiww000}
\psi_{\alpha,\alpha'}(w_{\rho, \sigma ,p}\otimes w'_{0,0,0})=
q^{-(\rho -\sigma +n-2p)\alpha' -(\alpha +\sigma+p)n' -2\alpha'(\alpha +\sigma+p) } w'_{0,0,0}\otimes w_{\rho, \sigma ,p}  +\sum_i c_ix_i\otimes y_i.
\end{equation}
where each $x_i$ is a basis element in $\{w'_{\rho', \sigma' ,p'}\}$ not equal to $w'_{0,0,0}$.   
\end{lemma}
\begin{proof}
Recall the forgetful functor from $\catdH$ to $\catd$.  Lemma \ref{L:El} and Corollary \ref{C:VincatdH} imply that $V^H(n,\alpha)$ maps to  $ V(n,\tilde{\alpha})$ under this functor.     Similarly, $V^H(n',\alpha')$ maps to  $ V(n',\tilde{\alpha}')$.
Now the braiding $c_{ V^H(n,\alpha),V^H(n',\alpha')}$ of  Theorem \ref{T:braiding} 
under the forgetful functor gives the desired isomorphism $\psi_{\alpha,\alpha'}$ in $\catd$.  

We have 
$$\psi_{\alpha,\alpha'}(w_{\rho, \sigma ,p}\otimes w'_{\rho', \sigma' ,p'})=\tau(\check{R}\mathcal{K}(w_{\rho, \sigma ,p}^\alpha\otimes {w'}_{\rho', \sigma' ,p'}^{\alpha'}))$$
where 
$$\mathcal{K}(w_{0,0,0}^\alpha\otimes {w'}_{\rho', \sigma' ,p'}^{\alpha'})=q^{-n_1(\alpha'+\sigma'+p'
)-\alpha(\rho'-\sigma' +n'-2p')-2\alpha(\alpha'+\sigma'+p')}w_{0,0,0}^\alpha\otimes {w'}_{\rho', \sigma' ,p'}^{\alpha'} $$
 Since $E_1w_{0,0,0}^\alpha=E_2w_{0,0,0}^\alpha=0$ it follows that  $\check{R}(w_{0,0,0}^\alpha\otimes {w'}_{\rho', \sigma' ,p'}^{\alpha'})=w_{0,0,0}^\alpha\otimes {w'}_{\rho', \sigma' ,p'}^{\alpha'}$ and the first formula in the lemma holds.   
 
 To prove the second formula, recall from Equation \eqref{E:TruR} that 
 $$\check{R}=1\otimes 1 + \sum_i d_i a_i \otimes b_i$$
 where each $b_i$ is of the form $F_1^kF_3^sF_2^t$ where at least one of the indices $k,s$ or $t$ is non-zero.  Therefore, from the defining relations of Theorem \ref{T:Structure} we have $b_i{w'}^{\alpha'}_{0,0,0}$ is a linear combination of basis vectors $w_{\rho', \sigma' ,p'}^{\alpha'}$ where $\rho', \sigma' ,p'$ are not all zero (since the action of either $F_1$ or $F_2$ on any basis vector increase at least one of the indices of the vector, see Equation \eqref{E:ActionVntildea_F}).    
Combining the above we have 
 \begin{align*} 
 \check{R}(w_{\rho, \sigma ,p}^{\alpha}\otimes {w'}_{0,0,0}^{\alpha'})&=w_{\rho, \sigma ,p}^{\alpha}\otimes {w'}^{\alpha'}_{0,0,0}+ \sum_i d_i( a_i \otimes b_i )(w_{\rho, \sigma ,p}^{\alpha}\otimes {w'}^{\alpha'}_{0,0,0})\\
 &=w_{\rho, \sigma ,p}^{\alpha}\otimes {w'}_{0,0,0}^{\alpha'}+ \sum_j d_j' y_j\otimes x_j
 \end{align*}
where each $x_i$ is a basis element in $\{{w'}^{\alpha'}_{\rho', \sigma' ,p'}\}$ not equal to ${w'}^{\alpha'}_{0,0,0}$.   Thus, since $\mathcal{K}$ acts diagonally on the basis, we just need to compute $\mathcal{K}(w_{\rho, \sigma ,p}^{\alpha}\otimes {w'}^{\alpha'}_{0,0,0})$.  This can be done as above to obtain Equation \eqref{E:psiww000}.   
\end{proof}
\begin{remark}
Clearly, the isomorphism $\psi_{\alpha,\alpha'}$ in Lemma \ref{comm} depends on the choice of $\alpha $ and $ \alpha'$.
\end{remark}

\begin{lemma}\label{epsilon}
Consider $\tilde{\alpha}_1,..., \tilde{\alpha}_n\in \C/{\rt}\Z$ with $ \bar{\alpha}_i \not\in \XX, i\in \{1,...,n\}$.
From the Lemma \ref{L:i,j}, there exists $ i,j$ such that
$ \bar{\alpha}_i+\bar{\alpha}_j\notin\{\bar 0, \bar{\frac{\rt}{2}}\}.$
Then, for any $\epsilon \in \C/\rt\Z$ such that: 
 $\bar{\alpha}_j-\bar{\epsilon}\notin\{\bar 0, \bar{\frac{\rt}{2}}\}$
and  
$\bar{\alpha}_{i}+\bar{\epsilon}\notin\{\bar 0, \bar{\frac{\rt}{2}}\} $
we can modify the weights without changing the tensor product in the following way:\\
$$V(0,\tilde{\alpha}_1)\otimes...\otimes V(0,\tilde{\alpha}_i)\otimes ...\otimes V(0,\tilde{\alpha}_j)\otimes  ...\otimes V(0,\tilde{\alpha}_{n})\simeq$$
$$V(0,\tilde{\alpha}_1)\otimes...\otimes V(0,\tilde{\alpha}_i+\epsilon )\otimes ...\otimes V(0,\tilde{\alpha}_j-\epsilon )\otimes  ...\otimes V(0,\tilde{\alpha}_{n}).$$
\end{lemma}
\begin{proof}
From the choice of $\epsilon $, Lemma \ref{CorDecomp} implies 
$$V(0,\tilde{\alpha}_{i})\otimes V(0,\tilde{\alpha}_{j})\simeq V(0,\tilde{\alpha}_{i}+\epsilon)\otimes V(0,\tilde{\alpha}_{j}-\epsilon).$$  
Combining this isomorphism with Lemma \ref{comm} we have the following isomorphisms:  
$$V(0,\tilde{\alpha}_1)\otimes...\otimes V(0,\tilde{\alpha}_i)\otimes ...\otimes V(0,\tilde{\alpha}_j)\otimes  ...\otimes V(0,\tilde{\alpha}_{n})$$ 
$$\simeq V(0,\tilde{\alpha}_1)\otimes...\otimes \hat{V}(0,\tilde{\alpha}_i)\otimes ...\otimes \hat{V}(0,\tilde{\alpha}_j)\otimes  ...\otimes V(0,\tilde{\alpha}_{n})\otimes V(0,\tilde{\alpha}_i)\otimes V(0,\tilde{\alpha}_j)$$ 
$$\simeq V(0,\tilde{\alpha}_1)\otimes...\otimes \hat{V}(0,\tilde{\alpha}_i)\otimes ...\otimes \hat{V}(0,\tilde{\alpha}_j)\otimes  ...\otimes V(0,\tilde{\alpha}_{n})\otimes V(0,\tilde{\alpha}_i+\epsilon )\otimes V(0,\tilde{\alpha}_j-\epsilon )$$
$$
\simeq V(0,\tilde{\alpha}_1)\otimes...\otimes V(0,\tilde{\alpha}_i+\epsilon )\otimes ...\otimes V(0,\tilde{\alpha}_j-\epsilon )\otimes  ...\otimes V(0,\tilde{\alpha}_{n}).
$$
This concludes the proof.  
\end{proof}
\section{The right trace and its modified dimension}  

\subsection{The existence of the right trace}
In Subsection \ref{SS:trace} we recalled several results about right traces.  Here we apply these results to construct a right trace on the  ideal generated by $V(0,\tilde{\alpha})$ for 
$\bar\alpha\notin \{\bar0, \bar{\frac{\rt}{2}}\}$.    

We've seen in the Decomposition Lemma \ref{L:DecompV0Vn} that for $\tilde{\alpha }, \tilde{\beta }\in \C/\rt\Z$ such that 
$\bar{\alpha },\bar{\beta }, \bar{\alpha }+\bar{\beta }\notin  \{\bar0, \bar{\frac{\rt}{2}}\}$
we have the following decomposition:
$$ V(0,\tilde{\alpha} )\otimes V(0,\tilde{\beta} )=V(0,\tilde{\alpha} +\tilde{\beta} )\oplus V(0,\tilde{\alpha} +\tilde{\beta}+1)\oplus V(1,\tilde{\alpha} +\tilde{\beta} ).$$
In the case 
$$V(0,\tilde{\alpha } )\otimes V(0,\tilde{\alpha } )^{\ast }=V(0,\tilde{\alpha } )\otimes V(0,-\tilde{\alpha } -1 ) $$ the decomposition is no longer semi-simple, and the two  $4$-dimensional modules corresponding to $V(0,-1)$ and $V(0,0)$ merge into an indecomposable non-simple $8$-dimensional module which we will denote by $V_1(\tilde{\alpha } )$.  More precisely we have the following result: 

\begin{proposition} \label{P:DecVotimesV*}
Let $\tilde{\alpha }\in \C/\Z$ such that 
$\bar{\alpha }\notin \{\bar0, \bar{\frac{\rt}{2}}\}$.  
 We have the following decomposition:
\begin{equation}\label{E:VV*VopV}
V(0,\tilde{\alpha } )\otimes V(0,\tilde{\alpha } )^{\ast} =V_1(\tilde{\alpha } ) \oplus V_2(\tilde{\alpha })
\end{equation} 
where 
$V_2(\tilde{\alpha } )$ is an  8-dimensional simple module
 and  $V_1(\tilde{\alpha } )$ is an 
 indecomposable 
 module
 such that $\Hom_\cat(\C,V_1(\tilde{\alpha } ))$ and $\Hom_\cat(V_1(\tilde{\alpha } ),\C)$ are both non-zero.  
 \end{proposition}
 
\begin{proof}
Recall $V(0,\tilde{\alpha } )^*$ is isomorphic to $V(0,-\tilde{\alpha } -1 )$.   Let  $\{w_{\rho,\sigma}^{\tilde{\alpha}}\}$ and $\{w_{\rho,\sigma}^{-\tilde{\alpha}-1}\}$ be the bases of $V(0,\tilde{\alpha } )$ and $V(0,-\tilde{\alpha } -1 )$ given in   Theorem~\ref{T:Structure}.  
Consider the vectors of $V(0,\tilde{\alpha } )\otimes V(0,-\tilde{\alpha } -1 )$:
$$v_7=q^{-\tilde{\alpha}-1} [\tilde{\alpha}] w^{\tilde{\alpha}}_{1,1} \otimes w^{-\tilde{\alpha}-1}_{0,0}+q^{\tilde{\alpha}} [\tilde{\alpha}+1] w^{\tilde{\alpha}}_{0,0} \otimes w^{-\tilde{\alpha}-1}_{1,1}$$
and
$$
 u_0=[\alpha]w_{0,0}^{\tilde{\alpha}}\otimes w^{-\tilde{\alpha}-1}_{1,0} + q^{-\tilde{\alpha}} [\alpha+1]w_{1,0}^{\tilde{\alpha}}\otimes w^{-\tilde{\alpha}-1}_{0,0} .
$$

Let  $V_1(\tilde{\alpha } )$ and $V_2(\tilde{\alpha } )$ be the modules generated by $v_7$ and $u_0$, respectively.  The action of these modules is given in Tables \ref{Table:ActionV1} and \ref{Table:ActionV2} where $\{v_i\}$ and $\{u_i\}$ are bases for the corresponding modules.  

\begin{table}[b]
\caption{Action on $V_1(\tilde{\alpha } )$, where $c=q^{-\tilde{\alpha }} (q^{-1}[\tilde{\alpha} ]-[\tilde{\alpha }+1 ])$.}\label{Table:ActionV1}
\begin{center}
 \begin{tabular}{ | l || c | c | c | c | c | c | c | c | }
    \hline 
      $V_1(\tilde{\alpha } )$& $v_0 \ \ \  $ & $v_1 \ \ \ $ & $v_2 \ \ \ $ & $v_3 \ \ \ $ & $v_4 \ \ \ $ & $v_5 \ \ \ $ & $v_6 \ \ \ $ & $v_7 \ \ \ $\\ \hline
      $E_1$ & $ \ \ 0 \ \ $ & $ \ \ 0 \ \ $ & $ \ \ v_1 \ \ $ & $0$ & $0$  & $v_4$ & $ \ \ 0 \ \ $ & $0$ \\ \hline
      $E_2$ & $0$ & $c\cdot v_0$ & $0$ & $0$ & $[\tilde{\alpha}][\tilde{\alpha}+1]v_3$ & $0$ & $v_5 $ & $-[\tilde{\alpha}][\tilde{\alpha }+1]v_2$ \\ \hline
      $F_1$ & $0$ & $v_2$ & $0$ & $0$ & $v_5$  & $0$ & $0$ & $0$ \\ \hline    
      $F_2$ & $v_1$ & $0$ & $v_3$ & $0$ & $0$  & $-c\cdot v_6$ & $0$ & $v_4$ \\ \hline    
   \hline
  \end{tabular}
\end{center}
\end{table}%

\begin{table}
\caption{Action on $V_2(\tilde{\alpha } )$, where $c=q^{-\tilde{\alpha }} (q^{-1}[\tilde{\alpha} ]-[\tilde{\alpha }+1 ])$.}\label{Table:ActionV2}
\begin{center}
  \begin{tabular}{ | l || c | c | c | c | c | c | c | c | }
    \hline 
      $V_2(\tilde{\alpha } )$& $u_0 \ \ \  $ & $u_1 \ \ \ $ & $u_2 \ \ \ $ & $u_3 \ \ \ $ & $u_4 \ \ \ $ & $u_5 \ \ \ $ & $u_6 \ \ \ $ & $u_7 \ \ \ $\\ \hline
      $E_1$ & $ \ \ 0 \ \ $ & $ \ \ u_0 \ \ $ & $ \ \ u_5 \ \ $ & $ \ \ u_6 \ \ $ & $ \ \ u_7 \ \ $ & $ \ \ 0 \ \ $  & $ (q+q^{-1})v_5$ & $ \ \ 0 \ \ $ \\ \hline
      $E_2$ & $0$ & $0$ & $0$ & $0$ & $-c\cdot u_3$ & $c\cdot u_0$ & $c\cdot u_1 $ & $c\cdot u_2$ \\ \hline
      $F_1$ & $u_1$ & $0$ & $u_3$ & $0$ & $0$  & $u_6$ & $0$ & $u_4$ \\ \hline    
      $F_2$ & $u_5$ & $u_2$ & $0$ & $u_4$ & $0$  & $0$ & $u_7$ & $0$ \\ \hline    
   \hline
  \end{tabular}
\end{center}
\end{table}%

We will show that module $V_1(\tilde{\alpha })$ is indecomposable.  Suppose $W_1$ and $W_2$ are modules such $V_1(\tilde{\alpha })=W_1\oplus W_2$.  Since $\{v_i\}$ is a basis of $V_1(\tilde{\alpha })$ there exists 
$$v=c_0v_0+ c_1v_1+ c_2v_2+...+c_7v_7
$$ 
such that $c_7\not= 0$ and $v\in W_1$ or $v\in W_2$.  Without loss of generality assume $v\in W_1$.  From Table \ref{Table:ActionV1} we have 
$ F_2E_2E_1E_2v$ is a non-zero multiple of $ v_1$.  So $v_1\in W_1$.    Then Table \ref{Table:ActionV1} implies that 
$$\{v_1, c^{-1}E_2v_1, F_1v_1, F_2F_1v_1 \}=\{v_0,v_1,v_2,v_3\}\subset W_1.$$
Similarly, $E_2F_2F_1F_2v$ is a non-zero multiple of $v_5$ so $v_5\in W_1 $ and 
$$\{v_5, E_1v_5, -c^{-1}F_2v_5 \}=\{v_4,v_5,v_6\}\subset W_1.$$
Since $W_1$ is a submodule we have
$$
v_7=c_7^{-1}(v-c_0v_0-c_1v_1 -c_2v_2-...-c_6v_6)\in W_1
$$
Thus, $W_1=V_1(\tilde{\alpha })$ and we have showed that $V_1(\tilde{\alpha })$ is indecomposable.

Next we will show that $V_2(\tilde{\alpha })$ is simple.  Suppose $U$ is a non-zero submodule of $V_2(\tilde{\alpha })$.  Notice that  the generator $u_0$ of $V_2(\tilde{\alpha })$ is a highest weight vector.  The idea is to push any non-zero vector of $U$ to a multiple of $u_0$.   So let $u$ be a non-zero vector of $U$.  Write $u$ in terms of the basis $\{u_i\}$:
$$
u=c_0u_0+c_1u_1+c_2u_2+...+c_7u_7.
$$ 
If there exists an element $x$ in $\Uq$ such that $xu$ is a non-zero multiple of $u_0$ then since $u_0$ is a generator of $V_2(\tilde{\alpha })$ we would have  $ U\cong V_2(\tilde{\alpha })$.   We will show this is true for all possible non-zero coefficients of $u$.  
\begin{enumerate}
\item If $c_4\not= 0$ then from the $\Uq$-action given in the above table we have $E_1E_2E_1E_2u$ is a non-zero multiple of $u_0$.  
\item If $c_4=0$ and $c_7\not=0$ then $E_2E_1E_2u$ is a non-zero multiple of $u_0$.  
\item If $c_4=c_7=0$ and $c_3\not=0$ then $E_1E_2E_1u$ is a non-zero multiple of $u_0$.  
 \item If $c_4=c_7=c_3=0$ and $c_6\not=0$ then $E_1E_2u$ is a non-zero multiple of $u_0$.  
   \item If $c_4=c_7=c_3=c_6=0$ and $c_2\not=0$ then $E_2E_1u$ is a non-zero multiple of $u_0$.  
\item If $c_4=c_7=c_3=c_6=c_2=0$ and $c_5\not=0$ then $E_2u$ is a non-zero multiple of $u_0$.  
\item Finally,  if $c_2=c_3=c_4=c_5=c_6=c_7=0$ and $c_1\not=0$ then $E_1u$ is a non-zero multiple of $u_0$.
\item Finally,  if $c_1=c_2=c_3=c_4=c_5=c_6=c_7=0$ then $c_0\not=0$ and $u$ is a non-zero multiple of $u_0$. 
\end{enumerate}
Thus, $U\cong V_2(\tilde{\alpha })$ and $V_2(\tilde{\alpha }) $ is simple.  

Next, we consider the head and socle of $V_1(\tilde{\alpha})$.  We have 
$$v_3=q^{2(-\tilde{\alpha}-1)}w^{\tilde{\alpha}}_{1,1} \otimes w^{-\tilde{\alpha}-1}_{0,0}-q^{-\tilde{\alpha}-1}w^{\tilde{\alpha}}_{0,1} \otimes w^{-\tilde{\alpha}-1}_{1,0}+q^{-\tilde{\alpha}}w^{\tilde{\alpha}}_{1,0} \otimes w^{-\tilde{\alpha}-1}_{0,1}+ w^{\tilde{\alpha}}_{0,0} \otimes w^{-\tilde{\alpha}-1}_{1,1} $$
which generates the trivial module in $V_1(\tilde{\alpha })$.  Thus, $\Hom_\cat(\C,V_1(\tilde{\alpha }))$ is non-zero.  Also, from Table \ref{Table:ActionV1} we can see the map
$$V_1(\tilde{\alpha })\to \C \text{ given by }c_0v_0+ c_1v_1+...+c_7v_7 \mapsto c_7$$
is a $\Uq$-module morphism.   Thus,  $\Hom_\cat(V_1(\tilde{\alpha }),\C)$ is non-zero.

Finally, we prove that Equation \eqref{E:VV*VopV} holds.  Since the dimension of $V(0,\tilde{\alpha } )\otimes V(0,\tilde{\alpha } )^{\ast}$ is equal to the sum of the dimensions of $V_1(\tilde{\alpha } )$ and $ V_2(\tilde{\alpha })$, 
 it suffices to show that $ V_1(\tilde{\alpha })\cap V_2(\tilde{\alpha })=\{0\}$.   
Suppose this is not true.  Then there exists a non-zero  $v\in V_1(\tilde{\alpha })\cap V_2(\tilde{\alpha })$.   Since $ V_2(\tilde{\alpha })$ is simple then $V_2(\tilde{\alpha })$ is isomorphic to the module $<v>$ generated by $v$.  But since $v\in V_1(\tilde{\alpha })$ then $V_2(\tilde{\alpha })\cong<v>\subset V_1(\tilde{\alpha })$.  Since $V_1(\tilde{\alpha })$ and $V_2(\tilde{\alpha })$ have the same dimension this implies that $V_2(\tilde{\alpha })\cong V_1(\tilde{\alpha })$ which is a contradiction because $V_1(\tilde{\alpha })$ contains the trivial module as a submodule  and $V_2(\tilde{\alpha })$ is simple.  Thus, we have the decomposition.   
\end{proof}

\begin{corollary}
Let $\tilde{\alpha} \in \C/ \rt\Z$ such that 
$\bar{\alpha }\notin \{\bar0, \bar{\frac\rt2}\}$.  
Then $V(0,\tilde{\alpha})$ is a right ambidextrous object in the category $\cat.$ 
\end{corollary} 
\begin{proof}
Equation \eqref{E:VV*VopV} gives a decomposition of $V(0,\tilde{\alpha } )\otimes V(0,\tilde{\alpha } )^{\ast}$ into indecomposable as in Equation \eqref{E:Decomp} where $W_1=V_1(\tilde{\alpha })$ and $W_2=V_2(\tilde{\alpha })$.
Since $W_2=V_2(\tilde{\alpha })$ is an 8-dimensional simple module then $\Hom_\cat(\C,W_2)=\Hom_\cat(W_2,\C)=0$.  From Lemma \ref{L:Uniquejj} there are unique $j,j'\in \{0,1\}$ such that $\Hom_{\cat}(\unit , W_{j})$ and $\Hom_{\cat}(W_{j'},\unit ) $ are non-zero.  Thus, $j=j'=1$ and Theorem \ref{T:RightAmbij=j'} implies $V(0,\tilde{\alpha } )$ is right ambidextrous.  
\end{proof}
\subsection{The modified trace}
 From Theorem 10 of \cite{GKP2} (for a statement see Theorem \ref{T:rigthambiTrace} above) we have that the right ambidextrous object $V(0,\tilde{\alpha})$ gives a unique right trace:
\begin{theorem}[\cite{GKP2}]
Let $\tilde{\alpha} \in \C/ \rt\Z$ such that 
$\bar{\alpha }\notin \{\bar0, \bar{\frac\rt2}\}$. 
There exists a non-zero right trace $\{\qt_V\}$ on the ideal $\ideal_{V(0,\tilde{\alpha})}$ which is unique up to multiplication by a non-zero scalar.
\end{theorem}
The following lemma shows that the ideal generated by $V(0,\tilde{\alpha})$ contains all objects of $\cat$ except the trivial module $\C$ and thus the right trace is defined on all these objects.  It follows that this ideal is independent of $\tilde{\alpha}$ and we will denote it by $\ideal$.  
\begin{lemma}\label{Generating module}
For any $\tilde{\alpha} \in \C/ \rt\Z$ such that $\bar{\alpha}\notin \{\bar0, \bar{\frac\rt2}\}$  we have 
$\ideal_{V(0,\tilde{\alpha})}=\cat\setminus \{\C\}.$ 
\end{lemma}
\begin{proof}
First, we will show $\ideal_{V(0,\tilde{\alpha})}\subseteq \cat\setminus \{ \C\}$.  By definition this ideal is contained in $\cat$ so we only need to show  $\C\notin  \ideal_{V(0,\tilde{\alpha})}$.  Suppose on the contrary that $\ideal_{V(0,\tilde{\alpha})}=\cat$.    From Lemma \ref{L:Uniquejj} and Theorem \ref{T:RightAmbij=j'} it follows that the trivial module $\C$ is right ambidextrous.  By Theorem \ref{T:rigthambiTrace} there is a unique right trace on  $\ideal_{\C}=\cat$.  It is easy to see this trace is equal to the usual quantum trace in $\cat$ and its associated modified dimension is the usual quantum dimension $\qdim$.  Since $\ideal_{V(0,\tilde{\alpha})}=\cat=\ideal_\C$ then from the proof of Lemma 4.2.2 in \cite{GKP} we have $\qdim(V(0,\tilde{\alpha}))\not= 0$ (note \cite{GKP} requires a braiding but it is easy to see the cited proof works in our pivotal context).  But this is a contradiction so we have the desired inclusion.  

Now, we will show the converse inclusion.   First, notice that if $\tilde{\beta}\in \C/\rt\Z$ satisfies $ \bar{\beta}, \bar{\alpha} + \bar{\beta}\notin \{\bar0, \bar{\frac\rt2}\}$ then Lemma \ref{L:DecompV0Vn} implies that $V(0,\tilde{\alpha}+\tilde{\beta})$ is a retract of $V(0,\tilde{\alpha})\otimes V(0,\tilde{\beta})$.  Therefore, $V(0,\tilde{\alpha}+\tilde{\beta})\in \ideal_{V(0,\tilde{\alpha})}$.  Now if $\tilde{\mu}\in \C/\rt\Z$ such that $\bar{\mu}\notin \{\bar0, \bar{\frac\rt2}\}$ then there exists $\tilde{\beta}, \tilde{\gamma} \in \C/\rt\Z$ such that 
$\bar{\beta}, \bar{\gamma} \notin \{\bar0, \bar{\frac\rt2}\}$, $\tilde{\mu}=\tilde{\alpha} + \tilde{\beta} + \tilde{\gamma}$ and 
$\tilde{\alpha} + \tilde{\beta} \notin \{\bar0, \bar{\frac\rt2}\}$. 
Lemma~\ref{L:DecompV0Vn} implies 
$V(0,\tilde{\mu}) $ is a retract of $V(0,\tilde{\alpha}+\tilde{\beta})\otimes V(0,\tilde{\gamma})$.  
We have proved that if $\tilde{\mu}\in \C/\rt\Z$ with $\bar{\mu}\notin \{\bar0, \bar{\frac\rt2}\}$ then  $V(0,\tilde{\mu}) \in \ideal_{V(0,\tilde{\alpha})}$.  

Now let $V\in \cat\setminus \{ \C \}$.  By definition of $\cat$ there exists $ \tilde{\alpha}_1,...,\tilde{\alpha}_n\in \C/  \rt\Z$ with $\bar{\alpha}_1,..., \bar{\alpha}_n \notin \XX $ such that $V$ is a retract of $V(0, \tilde{\alpha}_1)\otimes V(0, \tilde{\alpha}_2)\otimes...\otimes V(0, \tilde{\alpha}_n).$  Since $V(0, \tilde{\alpha}_1)$ is in the ideal $\ideal_{V(0,\tilde{\alpha})}$ then  $V\in \ideal_{V(0,\tilde{\alpha})}$.
\end{proof}

\subsection{Computations of modified dimensions}
 In the proof of Theorem \ref{T:Semisimple} we will show that the ideal $\ideal$ contains  $V(n,\tilde{\alpha})$ such that $0\leq n\leq\rt'-2$ and $\bar\alpha\notin \{\bar0, \bar{\frac\rt2}\}$.  
In the next lemma we will compute the modified quantum dimension of such modules.  Recall that the modified quantum dimension is defined to be $\qd(V)=\qt_V(\Id_V)$ for $V\in \ideal$. 
\begin{lemma}\label{L:CompOfdq}
If $V(n,\tilde{\alpha})\in \ideal$ for $0\leq n\leq\rt'-1$ and $\bar\alpha\notin \{\bar0, \bar{\frac\rt2}\}$ then the right trace $\{\qt_V\}_{V\in \ideal}$ can be normalized so
\begin{equation}\label{E:FromulaQDVnalpha}
\qd(V(n,\tilde{\alpha}))=\frac{\{n+1\}}{\{1\}\{\tilde\alpha\}\{\tilde\alpha+n+1\}}
\end{equation}
where $\{z\}=q^z-q^{-z}$.
\end{lemma}
\begin{proof}
In the proof of Lemma \ref{Generating module} we showed that $V(0,\tilde{\frac13})\in \ideal$.   The trace is unique up to a global scalar and  we choose a normalization so that 
$$
\qd(V(0,\tilde{\frac13}))=\qt_{V(0,\tilde{\frac13})}(\Id_{V(0,\tilde{\frac13})})=\frac{1}{\{\tilde{\frac13}\}\{\tilde{\frac13}+1\}}.
$$

Let $V(n,\tilde{\alpha})\in \ideal$ for $0\leq n\leq\rt'-1$ and $\bar\alpha\notin \{\bar0, \bar{\frac\rt2}\}$.  Fix $\alpha \in \C$ such that $\tilde{\alpha}=\alpha $ modulo $\rt\Z$.  
For $U,W\in \cat$ and $f\in \End_\cat(U\otimes W)$ define 
$$
\operatorname{ptr}^W(f) =\bp{(\Id_U \otimes \tev_W)(f\otimes \Id_{W^*})(\Id_U \otimes \coev_W)}.
$$
  Recall the isomorphisms $\psi_{\alpha,\frac13}$ and $\psi_{\frac13, \alpha}$ of Lemma \ref{comm}.  
  
Let  $S'_{\alpha,\frac13}$ and $ S'_{\frac13, \alpha}$ be the complex numbers defined by the following equations:  
$$
 S'_{\frac13, \alpha}\Id_{V(n,\tilde{\alpha})}=\operatorname{ptr}^{V(0,\tilde{\frac13})}\left(\psi_{\frac13, \alpha}\psi_{\alpha,\frac13}\right),\;\;\;
 S'_{\alpha,\frac13}\Id_{V(0,\tilde{\frac13})}=\operatorname{ptr}^{V(n,\tilde{\alpha})}\left(\psi_{\alpha,\frac13}\psi_{\frac13, \alpha}\right).
$$
Now from properties of the modified trace we have 
\begin{align*}
\qd(V(n,\tilde{\alpha})) S'_{\frac13, \alpha}&= \qt_{V(n,\tilde{\alpha})} \left(\operatorname{ptr}^{V(0,\tilde{\frac13})}\left(\psi_{\frac13, \alpha}\psi_{\alpha,\frac13}\right)\right)\\
&=\qt_{V(n,\tilde{\alpha})\otimes V(0,\tilde{\frac13})} \left(\psi_{\frac13, \alpha}\psi_{\alpha,\frac13}\right)\\
&=\qt_{V(0,\tilde{\frac13})\otimes V(n,\tilde{\alpha})} \left(\psi_{\alpha,\frac13}\psi_{\frac13, \alpha}\right)\\
&= \qt_{V(0,\tilde{\frac13})} \left(\operatorname{ptr}^{V(n,\tilde{\alpha})}\left(\psi_{\alpha,\frac13}\psi_{\frac13, \alpha}\right)\right)\\
&=\qd\Big(V\big(0,{\tilde{\frac13}}\big)\Big)S'_{ \alpha, \frac13}.
\end{align*}
Now, if $S'_{\frac13, \alpha}\neq 0$ (which we will show below) then 
\begin{equation}\label{E:ModQuDimEquation}
\qd(V(n,\tilde{\alpha})) =\frac{S'_{ \alpha, \frac13}}{ \left(\{\tilde{\frac13}\}\{\tilde{\frac13}+1\}\right)S'_{\frac13, \alpha}}.
\end{equation}
Thus, to find a formula for $\qd(V(n,\tilde{\alpha}))$ it suffices to compute $S'_{\frac13, \alpha}$ and $S'_{\alpha, \frac13}$. 

We now compute $S'_{\frac13, \alpha}$.  Let  $\{w_{\rho, \sigma ,p}\}_{\rho, \sigma\in\{0,1\}, p\in\{0,...n-1\}}$ and  $\{w'_{\rho', \sigma' ,0}\}_{\rho', \sigma'\in\{0,1\}}$ is the weight bases of $V(n,\tilde\alpha)$ and $V(0,\tilde{\frac13})$, respectively.    Any endomorphism of $V(n,\tilde{\alpha})$ maps the highest weight vector $w_{0,0,0} $ of $V(n,\tilde{\alpha})$ to a multiple of itself.  Since $V(n,\tilde{\alpha})$ is simple it is enough to compute this coefficient, in other words 
\begin{equation}\label{E:S'13alpha}
\operatorname{ptr}^{V(0,\tilde{\frac13})}\left(\psi_{\frac13, \alpha}\psi_{\alpha,\frac13}\right)(w_{0,0,0})=S'_{\frac13, \alpha}w_{0,0,0}.
\end{equation}
 Now  $\psi_{\frac13, \alpha}$ and $\psi_{\alpha,\frac13}$ are determined by the action of the $R$-matrix $\check{R}\mathcal{K}$ on the $\UqH$-modules $V^H(n,\alpha)$ and $V^H(0,{\frac13})$.  Since we are taking a partial trace only diagonal quantities of this action contribute when writing on the weight vector basis $\{w^{\frac13}_{\rho, \sigma ,0}\}_{\rho, \sigma\in\{0,1\}}$ of $V^H(0,{\frac13})$ given above.  So it is enough to know the values of  $\psi_{\alpha,\frac13}(w_{0,0,0}\otimes w'_{\rho', \sigma' ,0})$ and $\psi_{\frac13,\alpha}( w'_{\rho', \sigma' ,0}\otimes w_{0,0,0})$ which are computed in Lemma \ref{comm}.   Note, the terms $c_ix_i\otimes y_i$ in Equation \eqref{E:psiww000} are not diagonal and so do not contribute.  Thus, evaluating the left side of Equation \eqref{E:S'13alpha} we have
 \begin{align*}
 S'_{\frac13, \alpha}
&=
\sum_{\rho', \sigma'=0}^1q^{(-2n-4\alpha)(\frac13+\sigma')-2\alpha(\rho'-\sigma')} 
  (-1)^{\rho'+ \sigma'}{w'}_{\rho', \sigma' ,0}^*\left(K_2^{-2} w'_{\rho', \sigma' ,0}\right)\\
&=\sum_{\rho', \sigma'=0}^1q^{(-2n-4\alpha)(\frac13+\sigma')-2\alpha(\rho'-\sigma')-2(\tilde{\frac13 }+\sigma')} (-1)^{\rho'+ \sigma'}
 \\
&=q^{(-\frac23 -1)(2\alpha+n+1)}\left(q^{2\alpha +n +1}-q^{-n-1}-q^{n+1}+q^{-2\alpha -n-1}\right)\\
&= q^{-(\frac23 +1)(2\alpha+n+1)}\{\alpha\}\{\alpha + n +1\}.
\end{align*}
Similarly, 
$$
S'_{ \alpha, \frac13}=q^{-(2\alpha+n+1)(\frac23+1)}\frac{\{n+1\}}{\{1\}}\{1/3\}\{4/3\}.
$$
Finally, since $\{\tilde{x}\}=\{{x}\}$ for any $x\in \C$ then  Equation \eqref{E:ModQuDimEquation} implies the result.  
\end{proof}

\section{The relative $\C/\Z$-spherical category}
\subsection{Purification of $\cat$}\label{Purification}
The category $\cat$ that we've constructed still needs to be modified in order to obtain a relative $G$-spherical category.  One of the main problem is that there are an infinite number of non-isomorphic simple objects in each graded piece of $\cat$.   In order to obtain a finite number of objects in each grading, we will  ``purify'' the category using the modified trace.  This will have the effect of  removing all modules outside the alcove.
This generalizes the well known purification of a category discussed in Chapter XI of \cite{Tu}.

Let $V,W\in\ideal=\cat\setminus \{\C\}$.   A morphism $f\in \Hom_{\cat} (V,W) $ is called \emph{negligible} with respect to the right trace $\qt$ if
 $$\qt_W(f\circ g)=\qt_V(g\circ f)=0$$
 for all $ g\in \Hom_{\cat}(W,V)$.
Denote $\Negl(V,W)$ as the set of negligible morphisms from $V$ to $W$.  The set $\Negl(V,W) $ is actually a sub-vector space of $\Hom_{\cat}(W,V)$.    Thus, we can take the quotient and obtain a $\C$-vector space $\Hom_{\cat}(V,W)/\Negl(V,W).$
  We set $\Negl(V,\C)=\Negl(\C,V)=0$ for any $V\in \cat$.  

We describe a purification process of $\cat$ which will produce a category $\cat^N$ where all negligible morphism are zero.  We define a new pivotal $\C$-category $\cat^N$ whose objects are the same as in $\cat$.   The set of morphisms between two objects $V$ and $W$ of $\cat^N$ is
$$\Hom_{{\cat}^N}(V,W)=\Hom_{\cat}(V,W)/\Negl(V,W).$$
The composition, tensor product, pivotal structure and grading in $\cat^N$ is induced from $\cat$:

\begin{lemma}\label{L:catNGraded}
The category $\cat^N$ is a pivotal $\C$-category 
with a $\C/\Z$-grading induced from the grading of $\cat$.  
\end{lemma} 
\begin{proof}
First, we will show that $\cat^N$ is a pivotal $\C$-category. 
There is an obvious functor $\mathcal{F}: \cat \to \cat^N$ which is the identity on objects and maps a morphism to its class modulo negligible morphisms:
\begin{enumerate}
\item $\mathcal{F}(A)=A, \forall A \in Ob(\cat)$,
\item $\mathcal{F}(f)=[f] \in \Hom_{\cat^N}(A,B), \forall f \in \Hom_{\cat}(A,B)$. 
\end{enumerate}
Using this functor we can induce the tensor $\C$-linear structure of $\cat$ onto $\cat^N$.   We also define the dual structure on $\cat^N$ as the one coming  from $\cat$, via the functor $\mathcal{F}$. 
 Since the dualities morphisms in $\cat$ satisfy the compatibility conditions  for a pivotal structure, 
 then the corresponding dualities under $\mathcal{F}$ will also satisfy these compatibility conditions in $\cat^N$. 

Recall the definition of a $\Gr$-graded category given in Subsection \ref{SS:GSpher}.  
From Lemma \ref{L:catgraded}, we know that $\cat$ is $\C/\Z$-graded.  For any $g\in \C/\Z$, define $\cat^N_g:=\mathcal{F}(\cat_g)$.  It is easy to see this gives a $\C/\Z$-graded on $\cat^N$.    
\end{proof}
We will use the same notation for the object $V(n,\tilde{\alpha})$  of $\cat$ and the corresponding object  in $\cat^N$.

\begin{lemma} \label{ne}
If $W\in\cat$ such that $W$ is simple and $\qd(W)=0$ then every morphism to or from $W$ is negligible.   
\end{lemma}
\begin{proof}
It suffices to prove that $\qt_W(h)=0 $ for any $h\in \End_{\cat}(W)$.  
This will imply that if $V\in \cat$ and $f\in \Hom_{\cat}(V,W)$ then $$\qt_W(f\circ g)=0$$ for any $g\in \End_{\cat}(W,V).$ Thus, $f$ is negligible.    A similar statement holds for $f\in \Hom_{\cat}(W,V)$. 

To prove the first statement, let  $h\in \End_{\cat}(W)$.  
Since $W$ is simple, $\End_{\cat}(W)=\C \Id_W$ and we will define the scalar $<h>\in \C$ as the solution to the equation $h=<h> \Id_W$.  
But $\qd(W)=0$, in other words $\qt_W(\Id_W)=0$.  Thus, 
$$\qt_W(h)=\qt_W(<h> \Id_W)=<h> \qt_W(\Id_W)=0.$$
\end{proof}

\begin{lemma} 
Let $V$ and $W$ be objects in $\cat$ such that $W$ is simple and $\qd(W)=0$.   Then $V,W$ are also object in $\cat^N$ with the property that  the direct sum $V\oplus W$  is isomorphic to $V$ in $\cat^N$, in other words $V\oplus W\simeq_{\cat^N} V$.
\end{lemma} 
\begin{proof}
Let  $i_1 : V\rightarrow V \oplus W$ and $pr_1:V \oplus     W\rightarrow V$ be the  injection and projection morphisms with $pr_1\circ i_1=\Id_V$ in $\cat$.  This gives the relation   $pr_1\circ i_1=\Id_V$ in $\cat^N$.  
We want to show $i_1\circ pr_1=\Id_{V \oplus W}$ in $\cat^N$.  To do this consider the other inclusion and projection morphisms   
$i_2 : W\rightarrow V \oplus W$ and $pr_2:V \oplus     W\rightarrow W$ in $\cat$.
Then by definition  $\Id_{V \oplus W}=i_1\circ pr_1+i_2\circ pr_2$ in $\cat$.  But from Lemma \ref{ne} we have $i_2\circ pr_2$ is negligible.   Thus, in $\cat^N$ we have 
$\Id_{V \oplus W}=i_1\circ pr_1$ and so $i_1$ is the inverse of  $ pr_1$.  
\end{proof}
\begin{corollary} \label{L:negl}
 Let $\tilde{\gamma}\in \C/\rt\Z$ such that $\bar{\gamma}\notin \{\bar0,\bar{\frac\rt2}\}$.  If $V(\rt'-1,\tilde{\gamma})\in \ideal$ then for any $V\in \cat^N$ we have 
$V \oplus V(\rt'-1,\tilde{\gamma}) \simeq_{\cat^N} V.$ 
 In particular, $V(\rt'-1,\tilde{\gamma}) \simeq_{\cat^N} \{0\}$.  
\end{corollary}

\begin{proof}  From Lemma \ref{L:CompOfdq}, we have that $\qd(V(\rt'-1,\tilde{\gamma}))=0$.   Applying the previous lemma we conclude the isomorphism. 
\end{proof}

\subsection{Generically finitely semi-simple}
\begin{lemma}\label{L:simpleN}
Let $V$ be a simple object in $\cat$.  As an object of $\cat^N$, $V$ is either simple or $V\simeq_{\cat^N} \{0\}$.
\end{lemma}
\begin{proof}
Since $V$ is simple we have that $\End_{\cat}(V)\simeq \C \cdot \Id_V$ is the 1-dimensional vector space.  
By definition $$\End_{\cat^N}(V)=\End_{\cat}(V)/\Negl(V,V)= (\C \cdot \Id_V)/ \Negl (V,V).$$ 
Thus, $\End_{\cat^N}(V)$ is either 0 or 1-dimensional corresponding to the two cases of the lemma.   \end{proof}
\begin{lemma}\label{L:retract}
Let $U \in \cat$ such that $U=({_\cat}\oplus_{j\in J}S_j)\oplus W$ where $J$ is a finite indexing set and $S_j$ is simple for all $ j\in J$. Let  $V\in \cat$
be a retract of $U$ with maps $ i: V \rightarrow U$ and $p:U\rightarrow V$.
Then the following statements are true:

1) There exist $J' \subseteq J$ and $W' \subseteq W$ such that: $$Im(i)=(\oplus_{j\in J'}S_j)\oplus W'.$$ 
Moreover, if $i':V\rightarrow Im(i)$ is the function $i$ but with range $Im(i)$ then
$i'$ is an isomorphism with  inverse $p':=p |_{Im(i)}$.

2) $W'$ is a retract of $W$.

\end{lemma}
\begin{proof}
1) Denote by $p_j:U \rightarrow S_j$ and $p_W:U\rightarrow W$ the projections onto direct summands of $U$.  

Consider $J':= \{ j \in J | p_j\circ i \not=0 \}$ and $W'=Im(p_W \circ i)$.   
Since for any $j\in J$,  $S_j$ is simple then it is generated by any non-zero element.
Using this and the fact that $p_j \circ i:V\to S_j$ is a non-zero morphism for all $j\in J'$, we obtain that this morphism is surjective.
We conclude that $Im(i)={_\cat} \ (\oplus_{j\in J'}S_j)\oplus W'$.
So $i'$ is surjective and injective. 
Moreover, $p' \circ i'=p\circ i=\Id_{V}$.

2) We prove the second statement in two steps.

\textbf{Step 1.} We will show that $(\oplus_{j\in J'}S_j)\oplus W'$ is a retract of  $(\oplus_{j\in J}S_j)\oplus W$.
Consider  $\iota: (\oplus_{j\in J'}S_j)\oplus W' \rightarrow (\oplus_{j\in J}S_j)\oplus W$ the natural inclusion, of each component of the direct sum in the left to the corresponding one on the right hand side. From the first part of the proof we have $$p \circ \iota=p'.$$
Define $\pi: (\oplus_{j\in J}S_j)\oplus W \rightarrow (\oplus_{j\in J'}S_j)\oplus W'$ by $\pi:=i' \circ p.$ 
Then since $i'$ and $p'$ are inverses we have:
 $$ \pi \circ \iota= \ (i' \circ p) \circ \iota= \ i' \circ (p \circ \iota)= i' \circ p'= \Id_{(\oplus_{j\in J'}S_j)\oplus W'}.$$
This concludes the Step 1.

\textbf{Second 2.}
Consider  $\iota_{W'}: W' \rightarrow (\oplus_{j\in J'}S_j)\oplus W' $ and $\pi_{W'}: (\oplus_{j\in J'}S_j)\oplus W' \rightarrow W'$ 
 the injection and projection  with respect to the direct summand of $W'$.  Similarly, consider
the injection $\iota_{W}: W \rightarrow (\oplus_{j\in J}S_j)\oplus W $ and projection $\pi_{W}: (\oplus_{j\in J}S_j)\oplus W \rightarrow W$. 

Define $\iota': W' \rightarrow W$ and $\pi':W\rightarrow W'$ as:
 $$\iota':=\pi_W \circ \iota \circ \iota_{W'} \ \ \ and \  \ \ \pi':=\pi_{W'} \circ \pi \circ \iota_W. $$
By definition we have:
$$\pi' \circ \iota'= \pi_{W'} \circ \pi \circ (\iota_W \circ \pi_W )\circ \iota \circ \iota_{W'}.$$
Since $Im(\iota \circ \iota_{W'})\subseteq 0 \oplus W \subseteq (\oplus_{j\in J}S_j)\oplus W $, this means that $$(\iota_W \circ \pi_W )\circ \iota \circ \iota_{W'}=\iota \circ \iota_{W'}.$$
So, we obtain:
$$\pi' \circ \iota'= \pi_{W'} \circ \pi \circ \iota \circ \iota_{W'}.$$
Using the conclusion of the first step ($\pi \circ \iota=\Id$) we have:
 $$\pi' \circ \iota'= \pi_{W'} \circ \iota_{W'}=\Id_{W'}.$$
This finishes the proof of the second part. 
\end{proof}

\begin{lemma} \label{L:iso0}
Let $V, W\in \cat$ such that $W$ is retract of $V$ in $\cat$.  If $V$ is isomorphic to the zero module in $\cat^N$ (i.e. $V\simeq_{\cat^N} \{ 0 \}$)  then so is $W$:
$$W \simeq_{\cat^N} \{ 0 \}.$$
\end{lemma}
\begin{proof}
Let $i:W\rightarrow V$ and $\pi: V\rightarrow W$ be the retract in $\cat$.  Let $[i]:W\rightarrow V$ and $[\pi]: V\rightarrow W$ be there images in $\cat^N$.  
Since $V\simeq _{\cat^N} \{ 0 \}$, this means  the  zero maps in $\cat^N$:
$$[0]_V: V\rightarrow  \{ 0 \}  \ \ \ and \ \ \ [0]^V: \{ 0 \} \rightarrow V$$ 
are inverses of each other in  $\cat^N$. In particular, we have:
\begin{equation}\label{eqn:V00}
[0]^V \circ [0]_V=_{\cat^N} [\Id_V].
\end{equation}
Consider the zero maps: 
$$[0]_W: W\rightarrow  \{ 0 \}  \ \ \ and  \ \ \ [0]^W: \{ 0 \} \rightarrow W.$$
We have  $[0]_W \circ [0]^W=_{\cat^N} [0]=_{\cat^N}[\Id_{\{ 0 \}}]$.  For the other composition, notice that 
$$[0]_W=_{\cat^N}[0]_V\circ [i] \ \ \ and \ \ \ [0]^W=_{\cat^N}[\pi]\circ [0]^V$$ so we have:
$$[0]^W \circ [0]_W=_{\cat^N}([\pi] \circ [0]^V) \circ ([0]_V \circ [i])=_{\cat^N} [\pi] \circ ([0]^V \circ [0]_V )\circ [i]$$
$$=_{\cat^N} \pi \circ [\Id_V]\circ i=_{\cat^N} \pi \circ i=_{\cat^N} [\Id_W]$$
where third equality comes from Equation  \eqref{eqn:V00}.  Thus, we have shown $[0]_W$ and $[0]^W$ are inverses of each other.  
\end{proof}

\begin{theorem}\label{T:Semisimple}
 Let $g\in G , g\not\in \{ \bar{0}, \bar{\frac{\rt}{2}} \}$. 
 
 1) The category $\cat^N_g$ is semi-simple.  
 
 2) The set of isomorphism classes of simple objects in $ \bigcup_{g\in G\setminus \{\bar0,\bar{\frac{\rt}{2}}\}} \cat_g^N$ is
$$\left\{ V(n,\tilde{\gamma})| 0\leq n\leq\rt'-2, \; \tilde{\gamma}\in \C/\rt\Z, \;\bar{\gamma}\notin \Big\{\bar0,\bar{\frac{\rt}{2}}\Big\} \right\}.$$
\end{theorem}

\begin{proof}  \textbf{Proof of part 1).}
To prove the first statement we begin by showing that elementary tensor products of $V(0, \tilde{\alpha})$ which arrive in grading $g$ are semi-simple in $\cat^N_{g}$.  To do this we first work in $\cat$ and use induction on the number of terms in the tensor product. Then we show such a tensor product is semi-simple in $\cat^N_{g}$.  

Let $P(n)$ be the following statement:
\begin{quotation}
 If $h\in G$ and $\tilde{\alpha}_i \in \C/ \rt\Z$ such that $h\not \in \{ \bar{0}, \bar{\frac{\rt}{2}} \}$, $\bar{\alpha}_i\not\in \XX$ for $i=1,...,n$ and $\bar{\alpha}_1+\bar{\alpha}_2+\cdots + \bar{\alpha}_n=h$ then as an object in $\cat$ the tensor product 
$$V(0,\tilde{\alpha}_1)\otimes V(0,\tilde{\alpha}_2)\otimes ...\otimes V(0,\tilde{\alpha}_n)$$
can be written as a direct sum of modules of the following form:
 \begin{enumerate}[(a)]
\item $V(m,\tilde\beta)$ where $m\leq min\{n-1, {\rt}'-2 \}$ and $\bar\beta=h$,
\item $V(\rt'-1, \tilde{\delta}) \otimes W$ where $\bar{\delta}\notin \{ \bar{ 0}, \bar{\frac{\rt}{2}} \}$ and $W$ is an object of $\cat$.  
\end{enumerate}
Moreover, this decomposition contains at least one module of the form 
$V(min\{n-1, {\rt}'-2 \},\tilde{\beta})$ for some $\tilde{\beta}\in \C/\rt\Z$ with $\bar \beta=h$ and $W= \{ 0\}$ if $n<\rt'-1$.
\end{quotation}
We prove this statement by induction.  

\textbf{  The case $n=2$}.
Let $h \in G  \setminus \{ \bar{0}, \bar{\frac{\rt}{2}} \}$ and 
 $\tilde{\alpha}_1, \tilde{\alpha}_2 \in \C/ \rt\Z$ such that $\bar{\alpha}_1+\bar{\alpha}_2=h$.  
It follows from Lemma~\ref{L:DecompV0Vn} that:
$$V(0,\tilde{\alpha}_1)\otimes V(0,\tilde{\alpha}_2)=_{\cat}V(0,\tilde{\alpha}_1+\tilde{\alpha}_2 )\oplus V(1,\tilde{\alpha}_1+\tilde{\alpha}_2)\oplus V(0,\tilde{\alpha}_1+\tilde{\alpha}_2+1).$$
As we can see, all the modules have the right form $V(n, \tilde{\alpha})$ with $n\leq 1$ and there is one  $V(1,\tilde{\alpha}_1+\tilde{\alpha}_2)$ which occurs.  

\textbf{Next we assume $P(n)$ is true and show $P(n+1)$ holds.}
To do this we need to consider two cases $n\geq {\rt}'-1$ and $n< {\rt}'-1$. 

\textbf{Case 1: $n\geq {\rt}'-1$.}   
Let $\tilde{\alpha}_1,...,\tilde{\alpha} _{n+1}$ be as in the statement of $P(n+1)$.  
From Lemma \ref{L:i,j}, there exists $ i,j\in \{1,...,n+1\}$ such that $\bar{\alpha}_i+\bar{\alpha}_j\not \in \{ \bar{0}, \bar{\frac{\rt}{2}} \}.$  
Choose  $\epsilon \in \C/\rt\Z$ such that:  
\begin{enumerate}
\item $\bar{\alpha}_1+\bar{\alpha}_2+...+ \widehat{\bar{\alpha}}_j+...+\bar{\alpha}_{n+1}+\bar\epsilon\neq \bar{0}, \bar{\frac\rt2}, $
\item $\bar{\alpha}_j-\bar\epsilon\neq \bar{0}, \bar{\frac\rt2}, $
\item $\bar{\alpha}_{i}+\bar\epsilon\neq \bar{0}, \bar{\frac\rt2}. $
\end{enumerate} 
Using Lemma \ref{epsilon} and the Commutativity Lemma  \ref{comm}, we obtain that: 
\begin{align*}
V(0&,\tilde{\alpha}_1)\otimes...\otimes V(0,\tilde{\alpha}_i)\otimes ...\otimes V(0,\tilde{\alpha}_j)\otimes  ...\otimes V(0,\tilde{\alpha}_{n+1})\simeq_{\cat}\\
&V(0,\tilde{\alpha}_1)\otimes...\otimes V(0,\tilde{\alpha}_i+\epsilon )\otimes ...\otimes V(0,\tilde{\alpha}_j-\epsilon )\otimes  ...\otimes V(0,\tilde{\alpha}_{n+1})\simeq_{\cat} \\
&V(0,\tilde{\alpha}_1)\otimes...\otimes \widehat{V(0,\tilde{\alpha}_i)}\otimes ...\otimes \widehat{V(0,\tilde{\alpha}_j)}\otimes...\otimes V(0,\tilde{\alpha}_{n+1})\otimes V(0,\tilde{\alpha}_i+\epsilon)\otimes V(0,\tilde{\alpha}_j-\epsilon).
\end{align*}
This shows it suffices to prove that the statement $P(n+1)$ holds for 
weights of the form
$$(\tilde{\alpha}_1,...,\hat{\alpha}_i,...,\hat{\alpha}_j,...,\tilde{\alpha} _{n+1},\tilde{\alpha}_i+\epsilon, \tilde{\alpha}_j-\epsilon).$$
  
By the choice of $\epsilon$ we have $\tilde{\alpha}_1,...,\hat{\alpha}_i,...,\hat{\alpha}_j,...,\tilde{\alpha} _{n+1},\tilde{\alpha} _{i}+\epsilon$ has the total grading different than  $\bar{0}, \bar{\frac{\rt}{2}}$, so it satisfies the property in the statement of $P(n)$.  
Therefore, by the induction hypothesis there exists:
$$m_1,...,m_k\in  \{ 0,..., {\rt}'-3 \}, \tilde{\beta}_1,...,\tilde{\beta}_k, \tilde{\gamma}_1,...,\tilde{\gamma}_s, \tilde{\delta}_1,...,\tilde{\delta}_p \in \C/\rt\Z, \text{ and } W_1,...,W_p\in \cat$$
 such that 
\begin{align*}
V(0,\tilde{\alpha}_1)\otimes...\otimes \hat{V}(0,\tilde{\alpha}_i)\otimes ...\otimes \hat{V}(0,\tilde{\alpha}_j)\otimes  ...\otimes V(0,\tilde{\alpha}_{n+1})\otimes V(0,\tilde{\alpha}_i+\epsilon)\\
\cong_{\cat}(\oplus_u V(m_u,\tilde{\beta}_u))\oplus (\oplus_t V(\rt'-2,\tilde{\gamma}_t)) \oplus(\oplus_k(V(\rt'-1, \tilde{\delta}_k)\otimes W_k)) 
\end{align*} 
Taking the tensor product with $V(0,\tilde{\alpha} _{j}-\epsilon)$ we obtain:
\begin{multline*}
V(0,\tilde{\alpha}_1)\otimes...\otimes \hat{V}(0,\tilde{\alpha}_i)\otimes...\otimes \hat{V}(0,\tilde{\alpha}_j)\otimes...\otimes
V(0,\tilde{\alpha}_{n+1})\otimes V(0,\tilde{\alpha}_i+\epsilon)\otimes V(0,\tilde{\alpha}_{j}-\epsilon)\\
\cong_{\cat} (\oplus_u(V(m_u,\tilde{\beta}_u ) \otimes V(0,\tilde{\alpha}_{j}-\epsilon)) )\oplus (\oplus_t (V(\rt'-2,\tilde{\gamma}_t )\otimes V(0,\tilde{\alpha}_{j}-\epsilon)))\\
\oplus (\oplus_k ((V(\rt'-1, \tilde{\delta}_k)\otimes W_k)\otimes V(0,\tilde{\alpha}_{j}-\epsilon))).
\end{multline*}
Since the tensor product preserves the grading we have  
$$\bar{\beta}_u=\bar{\gamma}_t=\bar{\alpha}_1+...+\hat{\alpha}_j+\bar{\alpha}_{n+1}+\bar\epsilon$$
for all $ u\in\{ 1,...,k \}$ and $ t\in\{ 1,...,s \}$.  
It follows that $$ \bar{\beta}_u+ \bar{\alpha}_j-\bar\epsilon= \bar{\alpha}_1+...+\bar{\alpha}_{n+1}\not \in \{ \bar{0}, \bar{\frac\rt2} \}.$$  
Similarly,  $\bar{\gamma}_t+\bar{\alpha}_j-\bar\epsilon \not\in \{ \bar{0}, \bar{\frac\rt2} \}.$  
 Lemma \ref{L:DecompV0Vn} implies that the final expression in the previous tensor decomposition is isomorphic to 
\begin{align*} 
 \oplus_u\big( \ V(m_u,\tilde{\beta} _u&+\tilde{\alpha} _{j}-\epsilon) \oplus V(m_u+1,\tilde{\beta}_u+\tilde{\alpha} _{j}-\epsilon)\oplus \\
& \oplus V(m_u-1,\tilde{\beta}_u +\tilde{\alpha}_j-\epsilon+1) 
 \oplus V(m_u,\tilde{\beta}_u+\tilde{\alpha}_j-\epsilon+1)\ \big)\oplus  \\
& \oplus  \oplus_t \big(\ V(\rt'-2,\tilde{\gamma}_t+\tilde{\alpha}_j-\epsilon)\oplus (V(\rt'-1,\tilde{\gamma}_t+\tilde{\alpha}_j-\epsilon)\otimes {\unit})\oplus \\
&\oplus V(\rt'-3,\tilde{\gamma}_t+\tilde{\alpha}_j-\epsilon+1)\oplus V(\rt'-2,\tilde{\gamma}_t+\tilde{\alpha}_j-\epsilon+1)\ \big) \oplus \\
& \oplus \oplus_k \big(\ V(\rt'-1, \tilde{\delta}_k)\otimes W'_k \ \big)
 \end{align*} 
 where $W'_k=W_k\otimes V(0,\tilde{\alpha}_{j}-\epsilon)$. 
We notice that from the induction hypothesis $\bar{\delta}_k \notin \{  \bar {0}, \bar {\frac{\rt}{2}}\}$ and from the previous relation $\bar{\gamma}_t+\bar{\alpha}_j-\bar\epsilon \not\in \{ \bar{0}, \bar{\frac\rt2} \}$, so all the second components that occur are not in $\{ \bar{0}, \bar{\frac{\rt}{2}} \}$.  Also, notice that the decomposition contains $V(\rt'-2,\tilde{\gamma}_t+\tilde{\alpha}_j-\epsilon+1)$ as a summand.   Thus, we proved the step $P(n+1)$ in this case.

\textbf{Case 2: $n< {\rt}'-1$.} The proof of the previous case also works here except that things are slightly simpler in this case because no module of the form $V(\rt'-1,\tilde{\gamma})\otimes W$ appears in the large tensor product.  We highlight the differences: the first part of the proof is the same.  Then the induction hypothesis implies there exists:
$$m_1,...,m_k\in  \{ 0,..., n-1 \}, \text{ and } \tilde{\beta}_1,...,\tilde{\beta}_k\in \C/\rt\Z $$
 such that 
\begin{equation*} 
V(0,\tilde{\alpha}_1)\otimes...\otimes \hat{V}(0,\tilde{\alpha}_i)\otimes ...\otimes \hat{V}(0,\tilde{\alpha}_j)\otimes  ...\otimes V(0,\tilde{\alpha}_{n+1})\otimes V(0,\tilde{\alpha}_i+\epsilon)\cong \oplus_u V(m_u,\tilde{\beta}_u)
\end{equation*}
where at least one $m_i=n-1$.  
As above take the tensor product with $V(0,\tilde{\alpha} _{j}-\epsilon)$ then the Decomposition Lemma~\ref{L:DecompV0Vn} implies 
\begin{equation}\label{E:muboae}
V(m_u,\tilde{\beta}_u)\otimes V(0,\tilde{\alpha} _{j}-\epsilon)
\end{equation}
decomposes into a direct sum of modules of the form $V(m,\tilde\beta)$ where $m\leq m_u+1\leq n< \rt'-1$ and $\bar\beta=\bar{\beta}_u + \bar{\alpha}_{j}-\bar\epsilon\notin \{\bar0,\bar{\frac\rt2}\}$.  Also notice that when $m_u=m_i=n-1$ then the tensor product in Equation \eqref{E:muboae} as a summand for the form $V(m_i+1,\tilde\beta)=V(n,\tilde\beta)$.  Thus, we have proved that the statement for $P(n+1)$ holds.  

Now we will show that \textbf{ $\cat^N_g$ is semi-simple} 

Let $V \in \cat^N_g$. Then, from the definition, $V$ is a $\cat$-retract of a module
\begin{equation}\label{E:V0a1V0n}
V(0, \tilde{\alpha}_1)\otimes V(0, \tilde{\alpha}_2)\otimes...\otimes V(0, \tilde{\alpha}_n)  
\end{equation}
where $\bar{\alpha}_i \notin \XX$
 and 
$\bar{\alpha}_1+\bar{\alpha}_2+\cdots + \bar{\alpha}_n=g.$      
From the first part, we know that there exist $$m_1,...,m_k\in  \{ 0,..., {\rt}'-2 \}, \tilde{\beta}_1,...,\tilde{\beta}_k, \tilde{\delta}_1,...,\tilde{\delta}_p \in \C/\rt\Z,  \text{ and } W_1,...,W_p\in \cat$$
 such that 
\begin{align*}
V(0,\tilde{\alpha}_1)\otimes  ...\otimes V(0,\tilde{\alpha}_{n})\cong_{\cat}
\left(\bigoplus_u V(m_u,\tilde{\beta}_u)\right)\oplus\left(\bigoplus_t\left(V(\rt'-1,\tilde{ \delta}_t)\otimes W_t\right)\right) 
\end{align*} 
where  $\bar{\beta}_u, \bar{\delta}_t \not\in \{ \bar{0}, \bar{\frac{\rt}{2}}\}$.   
Applying Lemma \ref{L:retract} to the right side of previous equation there exists a subset $J' \subset \{1,...,k\} $ and a retract $ W'$ of $\oplus_t(V(\rt'-1, \tilde{\delta}_t)\otimes W_t)$   such that 
$$V\simeq_{\cat} (\oplus_{u \in J'} V(m_u,\tilde{\beta}_u))\oplus W'.$$

Now, for all  $  t \in \{1,...,p \},$ Corollary \ref{L:negl} implies 
$$V(\rt'-1, \tilde{\delta}_t)\simeq _{\cat^N} \{ 0 \}.$$
This shows that $$\bigoplus_t\left(V(\rt'-1, \tilde{\delta}_t)\otimes W_t\right)\simeq _{\cat^N}\{0\}.$$
Using \ref{L:iso0}, we obtain that:
$$W'\simeq _{\cat^N}\{0\}$$
Thus we conclude that in $\cat^N$: 
\begin{equation}
\label{eqn:ssv}
V\simeq \oplus_{u \in J'} V(m_u,\tilde{\beta}_u).
\end{equation}
Since the modules in the previous decomposition are simples in $\cat$, then Lemma~\ref{L:simpleN} implies they are also simple in $\cat^N$ and we obtain that $V$ is semi-simple in $\cat^N$.

\textbf{Proof of part 2).} Now we will prove the second part of the theorem.
Let $V\in \cat_g^N$ be a simple object (i.e. $\End_{\cat^N}(V)=\C\Id_{V}$), with $g\in G, g\notin \{ \bar{0}, \bar{\frac{\rt}{2}} \}.$   
By definition $V$ is obtained from a $\cat$-retract of tensor products of modules of the form $V(0,\alpha)$ and from Equation  \eqref{eqn:ssv} we have
\begin{equation*}
V\simeq_{\cat^N} \oplus_{u \in J'} V(m_u,\tilde{\beta}_u)
\end{equation*}
where each $V(m_u,\tilde{\beta}_u)$ is simple in both $\cat$ and $\cat^N$.  If the carnality of $J'$ was strictly greater than one then $dim(\End_{\cat^N}({V}))\geq 2$ which is a contradiction.  So 
$$V\simeq_{\cat^N} V(m_u,\tilde{\beta}_u)$$
for some $0\leq m_u\leq \rt'-2$ and $\bar{\beta_i}=g$. This shows that any simple object that occur in $\cat^N_g$ is of the desired form.  

For the other inclusion, let $0 \leq s \leq  {\rt}'-2 , \tilde{\gamma}\in \C/ \rt\Z, \bar{\gamma}=g\notin \{\bar0,\bar{\frac{\rt}{2}}\}$. 
We will show that $V(s,\tilde{\gamma})$ is in $\cat^N_g$.
   
    In the first part of this proof we showed the statements $P(n)$ hold.  The last part of these statements imply that for all $0 \leq m \leq  {\rt}'-2$ and $h \in \C/ \Z \setminus \{\bar0,\bar{\frac{\rt}{2}}\}$ there exists $\tilde\beta\in  \C/ \rt\Z$ such that $\bar\beta =h$ and $V(m,\tilde\beta)$ is a simple object in $\cat^N_h$.  
We use this as follows.  

Choose $\bar\beta\in \C/\Z$ such that $\bar \beta, \bar\gamma-\bar\beta\notin \{\bar0,\bar{\frac\rt2}\}$.  There exists a lift $\tilde\beta\in \C/\rt\Z$ of $\bar\beta$  so that $V(s,\tilde\beta)$ is $\cat^N_{\bar \beta}$ as discussed above.  Set $ \tilde\epsilon = \tilde{\gamma}-\tilde\beta$ then $\bar{\epsilon}$ and $ \bar\beta+\bar\epsilon=\bar{\gamma}$ are not in $\{  \bar{0}, \bar{\frac{\rt}{2}} \}$.

In $\cat$, by definition $V(s,\tilde\beta)$ is a retract of a module:  
\begin{equation*}
V(0, \tilde{\alpha}_1)\otimes V(0, \tilde{\alpha}_2)\otimes...\otimes V(0, \tilde{\alpha}_n).
\end{equation*}
Taking the tensor product of this module with $V(0,\tilde\epsilon)$, we get that $V(s,\tilde\beta)\otimes V(0,\tilde{\epsilon})$ is a $\cat$-retract of 
\begin{equation}\label{E:vvvveBB}
V(0, \tilde{\alpha}_1)\otimes V(0, \tilde{\alpha}_2)\otimes...\otimes V(0, \tilde{\alpha}_n)\otimes V(0,\tilde{\epsilon}).
\end{equation}
Since $\bar\beta, \bar{\epsilon}, \bar\beta+\bar\epsilon\notin\{  \bar{0}, \bar{\frac{\rt}{2}} \}$ we have
$$V(s,\tilde{\beta})\otimes V(0,\tilde\epsilon)\simeq_{\cat} V(s,\tilde{\gamma})\oplus V(s+1,\tilde{\gamma})\oplus (1-\tilde{\delta}_{s,0})V(s-1,\tilde{\gamma}+1)\oplus V(s,\tilde{\gamma}+1)$$ 
and we see that $V(s,\tilde{\gamma})$ is a $\cat$-retract of $V(s,\tilde{\beta})\otimes V(0,\tilde\epsilon)$.
Using properties of $\cat$-retracts (if $A$ is a retract of $B$ and $B$ is a retract of $C$, then $A$ is a retract of $C$) and the previous two decompositions, we have
$V(s,\tilde{\gamma})$ is a $\cat$-retract of 
the module in Equation \eqref{E:vvvveBB}.   
This concludes the proof.  
\end{proof}

 A set of simple objects $A$ is said to be \emph{represented} by a set  of simple objects $R_A$ if any element of $A$ is isomorphic to a unique element of $R_A$.   
Lemma \ref{L:catNGraded} and Theorem \ref{T:Semisimple} imply the following corollary.  
\begin{corollary}\label{C:GenFinSSPiv}
The category $\cat^N$ is a generically finitely $\C/\Z$-semi-simple 
  pivotal 
 $\C$-category with
  small
    symmetric subset  $\X=\frac12\Z/\Z$.  The   class of generic
  simple objects $\A$ of $\cat^N$ is represented by
\begin{equation}\label{E:AsimpleV}
R_{\A}=\{ V(n,\tilde{\gamma})\; | \; 0\leq n\leq\rt'-2, \tilde{\gamma}\in \C/\rt\Z, \bar{\gamma}\notin \X \}.
\end{equation} 
\end{corollary}

\subsection{Trace}  Here we will show that the right trace $\qt$ on $\ideal$ induces a trace in $\cat^N$.  
\begin{lemma}
 The full subcategory $\ideal^N:= \cat^N\setminus \{  \C \} $ is a right ideal in $\cat^N$.
\end{lemma}
\begin{proof}
We need to show that $\ideal^N$ satisfies the two conditions to be a right ideal (see Subsection \ref{SS:trace}).  The first condition is true from the definitions of $\cat$ and $\cat^N$.  For the second condition we need to check that the trivial object $\C$ is not a retract of an object in $\cat^N$.  On the contrary, suppose there exists an object $W$ in $\cat^N$ and morphisms $f\in \Hom_{\cat^N}(\C, W)$ and $g\in \Hom_{\cat^N}(W,\C)$ such that $gf =\Id_{\C}$.  But by definition   $\Hom_{\cat^N}(\C, W)=\Hom_{\cat}(\C, W)$ and $ \Hom_{\cat^N}(W,\C)=\Hom_{\cat}(W,\C)$ so $f$ and $g$ give a $\cat$-retract of the trivial module which would imply that $\ideal=\cat$ which is a contradiction to Lemma \ref{Generating module}.  Thus, $\C\notin\cat^N$.
\end{proof}

\begin{lemma}
For $V\in  \ideal^N$ the assignment $ \qtN_V:\End_{\cat^N}(V)\to \C$ given by $[f]\mapsto \qt_V(f)$ is a well defined linear function.  Moreover, the family $\{\qtN_V\}_{V\in  \ideal^N}$ is a right trace on $\ideal^N$.  
\end{lemma}
\begin{proof}
We need to show $\qtN_V$ does not depend on the representative of $[f]$ in
$$\End_{\cat^N}(V)=\Hom_{\cat^N}(V,V)=\Hom_{\cat}(V,V)/Negl(V,V).$$
Suppose $[f]=[g]$ then $f=g+h$ for some $h\in Negl(V,V)$.  Then $ \qt_V(f)= \qt_V(g+h)=\qt_V(g)$, implying $\qtN_V([f])=\qtN_V([g])$.  

In order to prove that $\qtN_V$ is a right trace on $\ideal^N$, we have to prove that this satisfies the conditions $1)$ and $2)$ from the definition.

1) Let $U,V\in \ideal^N$ and $[f]\in \Hom_{\cat^N}(V,U), [g] \in \Hom_{\cat^N}(U,V)$. Let $f\in \Hom_{\cat}(V,U), g \in \Hom_{\cat}(U,V)$ such that the class of $f$ and $g$ in $\cat^N$ are $[f]$ and $[g]$ respectively. Then
$$\qtN_V([g] [f])=\qtN_V([g f])= \qt_V(g f)=\qt_U(f  g)=\qtN_U([f g])=\qtN_U([f] [g]).$$
2) Consider $U\in \ideal^N$ and $W\in \cat$ and $f\in \End_\cat^N(U\otimes W)$. Let $f \in \End_\cat(U\otimes W)$ such that the class of $f$ in $\cat$ is $[f]$.
Then we obtain:  
$$ \qtN_{U\otimes W}\bp{[f]}=\qt_{U\otimes W}\bp{f}=\qt_U \bp{(\Id_U \otimes \tev_W)(f\otimes \Id_{W^*})(\Id_U \otimes \coev_W)}=$$
$$=\qtN_U \bp{([\Id_U] \otimes [\tev_W])([f]\otimes [\Id_{W^*}])([\Id_U] \otimes [\coev_W)]}$$
The previous two equalities conclude the statement.
\end{proof}

\subsection{T-ambi pair}
Let $\qd:R_\A\to \C$ be the function given in Equation   \eqref{E:FromulaQDVnalpha}, in other words:
\begin{equation}\label{E:FromulaQDVnalpha22}
\qd(V(n,\tilde{\alpha}))=\frac{\{n+1\}}{\{1\}\{\tilde\alpha\}\{\tilde\alpha+n+1\}}
\end{equation}
for   $V(n,\tilde{\alpha})\in R_\A$.  Extend this function to $\A$ by requiring $\qd(V)=\qd(V(n,\tilde{\alpha}))$ if $V$ is isomorphic to $V(n,\tilde{\alpha})$.   
\begin{lemma}\label{L:tambiPairCatN}
The pair $(\A,\qd)$ is a t-ambi pair in $\cat^N$.  
\end{lemma}
\begin{proof}
Consider $\qdN$ the modified dimension on $\ideal^N$ coming from the right trace $\qtN$ on $\cat^N$.

Let $\B:= \{ V \in \ideal^N \bigcap (\ideal^N)^{*}\mid V  \text{ simple, } \qdN(V)=\qdN(V^{*})  \}$.
From Theorem \ref{T:Bt-ambi}, it follows that $(\B, \qd^N)$ is a t-ambi pair.
We notice that $(\ideal^N)^{*}=\ideal^N$.

We will prove that $\A \subseteq \B$ and that $\qdN$ is determined by Equation \eqref{E:FromulaQDVnalpha22}.
Let $V\in \A$.  By definition there exists $ 0\leq n\leq\rt'-2$ and $ \tilde{\gamma}\in \C/\rt\Z$ with $\bar{\gamma}\notin \X$ such that $V\simeq V(n, \tilde{\gamma})$.  We have
$$\qdN(V(n, \tilde{\gamma}))=\qtN([\Id_V(n, \tilde{\gamma})])= \qt(\Id_V(n, \tilde{\gamma}))=\qd(V(n, \tilde{\gamma}))$$ 
and so $\qdN(V(n, \tilde{\gamma})) $ is given by the formula in  Equation \eqref{E:FromulaQDVnalpha22}.

Since $ V(n,\tilde{\alpha } )^{\ast }=V(n, -\tilde{\alpha }-\tilde{n}-\tilde{1})$, Equation \eqref{E:FromulaQDVnalpha22} implies
$$
\qd(V(n,\tilde{\alpha}))=\qd((V(n,\tilde{\alpha}))^*).
$$
We conclude that
$$\qdN(V(n,\tilde{\alpha}))=\qdN((V(n,\tilde{\alpha}))^*)$$ for any $V \in \A$.
This shows that $\A \subseteq \B$.   Thus, since $(\B, \qdN)$ is a t-ambi pair, it is easy to check that  $(\A, \qd)$ is a t-ambi pair.
\end{proof}

\subsection{The $\bb$ map} \label{SS:bbmap}
Here we show $\cat^N$ has a map $\bb$ as in the definition of a relative $G$-spherical category.  To do this we need the following technical lemmas.

\begin{lemma} \label{L:isomN}
Let $L,R\in \cat^N$ with $L\in \cat^N_g$, $R\in \cat^N_h$ with $g,h\notin\{ \bar{0}, \bar{\frac{\rt}{2}}\}$.  Suppose $L=_{\cat^N}L_1\oplus L_2$ and $R=_{\cat^N}R_1 \oplus R_2$ such that 
$$L\simeq_{\cat^N}R \;\text{ and } \; L_1\simeq_{\cat^N}R_1 .$$
Then 
$$L_2\simeq_{\cat^N}R_2 .$$
\end{lemma}
\begin{proof}
Using Theorem \ref{T:Semisimple}, we have that both $L$ and $R$ are semi-simple in $\cat^N$. More precisely, there exists $ N \in \N,$ and $ S_1,...,S_N \in R_\A$ all different such that:
$$L=\eta_1 S_1\oplus...\oplus \eta_N S_N$$
$$R=\eta'_1 S_1\oplus...\oplus \eta'_N S_N\oplus J$$ where $\eta_i,\eta'_i \in \N,$ are the multiplicities of the simple object $S_i$ and $J$ is a direct sum of elements of  $R_A$ which are all different than $ S_i, i\in \{1,...,N \} $. 
As an observation, from the computation of $\qd$ we have:
$$ \qd(V)\not=0  \ \ \ \ \ \ \ \  \text{ for all }   \ \ \ \ \ \ \ \    V \in R_A,$$
in particular $\qd(S_i)\not=0, $ for all $ i \in \{1,...,N \}$.

We have that:
\begin{align*}\Hom_{\cat}(L,R)&=\Hom_{\cat}(\eta_1 S_1\oplus...\oplus \eta_N S_N, \eta'_1 S_1\oplus...\oplus \eta'_N S_N\oplus J)\\
&= \bigoplus_{i,j}(\Hom_{\cat}(\eta_i S_i, \eta'_j S_j))\oplus\bigoplus_{i}(\Hom_{\cat}(\eta_i S_i, J)).
\end{align*}
We notice that $\Hom(S_i,S_j)=0$ for $i \not=j$ and $\Hom(S_i,J)=0$ since J has no $S_i$-isotipic components so:
$$\Hom_{\cat}(L,R)=\bigoplus_{i}(\Hom_{\cat}(\eta_i S_i, \eta'_i S_i)).$$
Now we will study the negligible morphisms from this space.
By definition $\Negl(L,R)\subseteq \Hom_{\cat}(L,R)$ as a vector subspace.
From the last two relations we obtain that:
$$\Negl(L,R)=\bigoplus_{i}(\Negl(\eta_i S_i, \eta'_i S_i)).$$
We will prove that actually we have no negligible morphisms between isotipic components of $S_i$.  

Suppose that there exists $ f \in Negl (\eta_i S_i, \eta'_i S_i)$ which is non-zero.
For $k \in \{1,...,N \}$, denote by $\iota_k: S_i \rightarrow \eta_i S_i $ and $\pi_k: \eta'_i S_i\rightarrow S_i$ the inclusion and projection of the $k^{th}$ component.

Since $f$ is non-zero, then there exists $ k,l \in \{1,...,N \} $ such that $\pi_l \circ f \circ \iota_k\not=0$.  Also, since $S_i$ is simple in $\cat$: 
$$\pi_l \circ f \circ \iota_k=\pi_l \circ f \circ \iota_k(1) \Id_{S_i}.$$  
At the level of the modified trace we have:
$$
\qt_{S_i}(\pi_l \circ f \circ \iota_k)
=(\pi_l \circ f \circ \iota_k(1))\qt_{S_i}(\Id_{S_i})
=(\pi_l \circ f \circ \iota_k(1))  \qd(S_i)\not=0.$$
From the properties of $\qt$, we have:
$$\qt_{S_i}(\pi_l \circ f \circ \iota_k)=\qt_{\eta'_i S_i}( f \circ \iota_k \circ \pi_l )=0$$
(since $f$ is negligible).  

The last two equalities lead to a contradiction. 
We conclude that $Negl(L,R)=\{ 0 \}$ and so:
$$\Hom_{\cat}(L,R)= \Hom_{\cat^N}(L,R).$$
Now let $[\phi] \in \Hom_{\cat^N}(L,R)$ be an isomorphism.
Consider $\phi \in \Hom_{\cat}(L,R)$ that gives $[\phi]$ in $\cat^N$.
From the previous considerations, $\phi: L \rightarrow R$ is an isomorphism in $\cat$. 

Using this, we obtain that $J= \{ 0 \}$ (it is not possible to have more isotipic components in $R$ than in $L$). Also, since we are in a category of representations which are semi-simple and morphisms between representations, we obtain that 
$$\eta_i=\eta'_i, \forall i\in \{ 1,...,N \}.$$
So now, both $R$ and $L$ are semi-simple modules in $\cat$ with the same isotipic decomposition.

Now both $L_1$ and $R_1$ are direct summands in $L$ and $R$. It means that each of them has a semi-simple decomposition with modules from the set $S_i$.
But $L_1$ and $R_1$ are isomorphic in $\cat^N$.
Using the same argument as in the first part with $L$ and $R$, we obtain that $L_1$ and $R_1$ are isomorphic in $\cat$.
It means that they have the same isotipic decompositions with the same multiplicities.
Let us compose $\phi$ to the right with an automorphism of $L$ that makes a permutation on the isotipic components such that the ones corresponding to $L_1$ are sent onto the ones corresponding to $\phi^{-1}(R_1)$ respectively.
This means that we obtain an isomorphism $$\tilde{\phi}:L\rightarrow R$$ such that $$\tilde{\phi}(L_1)=R_1.$$
We conclude that 
$$\tilde{\phi}|_{L_2}: L_2\rightarrow R_2 $$
is an isomorphism in $\cat$ and also in $\cat^N$.
\end{proof}     

 \begin{lemma}\label{L:DecmnINTO0m+n}
 For all $\tilde{\alpha}, \tilde{\beta} \in \C/\rt\Z$ and $n\in \N$ such that $\bar{\alpha}, \bar{\beta}, \bar{\alpha}+\bar{\beta} \notin \{ \bar{0}, \bar{\frac{\rt}{2}}  \}$ and $n\leq \rt'-2$ then given $m\leq n$ we have
 $$
 V(m,\tilde\alpha)\otimes  V(n,\tilde\beta)\simeq_{\cat^N} V(0,\tilde\alpha)\otimes \Big(V(n+m,\tilde\beta)\oplus   V(n+m-2,\tilde\beta+1)\oplus \cdots \oplus V(n-m,\tilde\beta+m)\Big)
 $$
 where we set $V(k,\tilde\beta)=0$ if $k\geq \rt'-1$.  
 \end{lemma}
 \begin{proof}
We will show the statement by induction on $m$.   The case $m=0$ is true from Lemma \ref{L:DecompV0Vn}.  Next, we will check the case $m=1$. 

Let $\tilde{\alpha}, \tilde{\beta} \in \C/\rt\Z$ and $n\in \N$ such that $\bar{\alpha}, \bar{\beta}, \bar{\alpha}+\bar{\beta} \notin \{ \bar{0}, \bar{\frac{\rt}{2}}  \}$ and $1\leq n\leq \rt'-2$.  Choose $\tilde\gamma\in \C/\rt\Z$  such that $ \bar{\gamma}, \bar{\alpha}- \bar{\gamma}, \bar{\alpha}- \bar{\gamma}+\bar{\beta} \notin \{ \bar{0}, \bar{\frac{\rt}{2}}  \}$.
From Lemma \ref{L:DecompV0Vn} we have 
\begin{multline*}
V(0,\tilde{\alpha }- \tilde\gamma)\otimes V(n,\tilde{\beta } )\simeq_{\cat^N}V(n,\tilde{\alpha }- \tilde\gamma+\tilde{\beta})\oplus V(n,\tilde{\alpha }- \tilde\gamma+\tilde{\beta}+1)\\
\oplus  (1- \delta_{\rt'-2,n}) V(n+1,\tilde{\alpha }- \tilde\gamma+\tilde{\beta})\oplus V(n-1,\tilde{\alpha }- \tilde\gamma+\tilde{\beta}+1)
\end{multline*}
in $\cat^N$. Take the tensor product of both sides of this equation with $V(0,\tilde{\gamma } )$.  Then decomposing the left side by grouping the first two simple modules together we have
\begin{multline*}
\Big(V(0,\tilde{\gamma } )\otimes V(0,\tilde{\alpha }-\tilde{\gamma} )\Big)\otimes V(n,\tilde{\beta } )\simeq_{\cat^N} \left( V(0,\tilde{\alpha} )\otimes V(n,\tilde{\beta } )\right)\\
\oplus \left(V(0,\tilde{\alpha} +1)\otimes V(n,\tilde{\beta } )\right)
\oplus \left( V(1,\tilde{\alpha}  )\otimes V(n,\tilde{\beta } )\right)
\end{multline*}
On the other hand the left side is
\begin{multline*}
\left(V(0,\tilde{\gamma } )\otimes V(n,\tilde{\alpha }-\tilde{\gamma }+\tilde{\beta})\right)
\oplus \left(V(0,\tilde{\gamma } )\otimes V(n,\tilde{\alpha }-\tilde{\gamma }+\tilde{\beta}+1)\right)\\
\oplus  (1- \delta_{\rt'-2,n}) \left(V(0,\tilde{\gamma } )\otimes  V(n+1,\tilde{\alpha }-\tilde{\gamma }+\tilde{\beta})\right)\\
\oplus \left(V(0,\tilde{\gamma } )\otimes V(n-1,\tilde{\alpha }-\tilde{\gamma }+\tilde{\beta}+1)\right)
\end{multline*}
Now, using Corollary \ref{CorDecomp} we see the first two tensor products in the last two expressions are the same direct sum of simple modules so they are isomorphic.  Thus,  Lemma \ref{L:isomN} implies
\begin{multline*}
  V(1,\tilde{\alpha}  )\otimes V(n,\tilde{\beta } )
\simeq_{\cat^N}
 (1- \delta_{\rt'-2,n}) \left(V(0,\tilde{\gamma } )\otimes  V(n+1,\tilde{\alpha }-\tilde{\gamma}+\tilde{\beta})\right)\\
\oplus \left(V(0,\tilde{\gamma } )\otimes V(n-1,\tilde{\alpha }-\tilde{\gamma}+\tilde{\beta}+1)\right).
\end{multline*}
Using Corollary \ref{CorDecomp} we see the right hand side of this equation is isomorphic to 
$$
(1- \delta_{\rt'-2,n}) \left(V(0,\tilde{\alpha } )\otimes  V(n+1,\tilde{\beta})\right)\\
\oplus \left(V(0,\tilde{\alpha } )\otimes V(n-1,\tilde{\beta}+1)\right).
$$
Thus, we have proved the lemma for the case $m=1$.  

Now assume the statement is true for $k\leq m$ and we will show the statement holds for $m+1$.   Let $\tilde\alpha,\tilde\beta,\tilde\gamma $ and $n$ be as above.  
Let us denote:
$$E^{\tilde{\beta}}_{m,n}:=\Big(V(n+m,\tilde\beta)\oplus   V(n+m-2,\tilde\beta+1)\oplus \cdots \oplus V(n-m,\tilde\beta+m)\Big).$$

From the induction hypothesis we have
\begin{equation}\label{E:gammamalphanbeta}V(0,\tilde{\gamma})\otimes V(m,\tilde{\alpha}-\tilde{\gamma})\otimes V(n,\tilde{\beta})\simeq_{\cat^N} V(0, \tilde{\gamma}) \otimes V(0,\tilde{\alpha}-\tilde{\gamma})\otimes E^{\tilde{\beta}}_{m,n}.
\end{equation}
Using Lemma \ref{L:DecompV0Vn} to decompose the first tensor product we have the left hand side of Equation \eqref{E:gammamalphanbeta} is isomorphic to
$$\left( V(m, \tilde{\alpha})\otimes V(n, \tilde\beta)\right) \oplus \left(V(m, \tilde{\alpha}+1)\otimes V(n, \tilde\beta)\right) \oplus$$
$$ \oplus \left(V(m-1, \tilde{\alpha}+1)\otimes V(n, \tilde\beta) \right)\oplus 
\left(V(m+1, \tilde{\alpha})\otimes V(n, \tilde\beta)\right)$$
Similarly, the right hand side of Equation  \eqref{E:gammamalphanbeta}  is isomorphic to 
$$ \left(V(0, \tilde{\alpha}) \otimes E^{\tilde{\beta}}_{m,n} \right)\oplus \left(V(0, \tilde{\alpha}+1) \otimes E^{\tilde{\beta}}_{m,n}\right)\oplus \left(V(1, \tilde{\alpha}) \otimes E^{\tilde{\beta}}_{m,n}\right) $$

From the induction hypothesis, the first two terms of the last two expressions are isomorphic, thus from Lemma~\ref{L:isomN} we obtain  
\begin{equation}\label{E:gghe77} \left(V(m-1, \tilde{\alpha}+1)\otimes V(n, \tilde\beta) \right)\oplus 
\left(V(m+1, \tilde{\alpha})\otimes V(n, \tilde\beta)\right) \simeq_{\cat^N} V(1, \tilde{\alpha}) \otimes E^{\tilde{\beta}}_{m,n}.
\end{equation}

Next, we decompose the right side of this equation.  By induction we know that for any $1\leq n'\leq \rt'-2$ and $k\leq m$ we have
\begin{equation*}\label{E:jj7766}
V(1,\tilde{\alpha}  )\otimes V(n',\tilde{\beta } +k )
\simeq_{\cat^N}
 V(0,\tilde{\alpha }) \otimes  \left( V(n'+1,\tilde{\beta}+k) \oplus V(n'-1,\tilde{\beta}+k+1)\right).
 \end{equation*}
Applying this to each term of the sum $E^{\tilde{\beta}}_{m,n}$ we obtain:
$$V(1, \tilde{\alpha}) \otimes E^{\tilde{\beta}}_{m,n}\simeq_{\cat^N} V(0,\tilde{\alpha })\otimes \left(E^{\tilde{\beta}}_{m,n+1}\oplus E^{\tilde{\beta}+1}_{m,n-1} \right)$$
From the definition $E^{\tilde{\beta}}_{m,n}$ we have 
 \begin{align*}V(0,\tilde{\alpha })\otimes & \left( E^{\tilde{\beta}}_{m,n+1}\oplus E^{\tilde{\beta}+1}_{m,n-1}\right)=\\ 
&=V(0,\tilde{\alpha })\otimes \left( E^{\tilde{\beta}}_{m,n+1}\oplus \left( E^{\tilde{\beta}+1}_{m-1,n} \oplus V(n-1-m, \tilde{\beta}+m+1)\right)\right)\\
&= V(0,\tilde{\alpha })\otimes \left(\left( E^{\tilde{\beta}}_{m,n+1}\oplus   V(n-1-m, \tilde{\beta}+m+1)\right)\oplus  E^{\tilde{\beta}+1}_{m-1,n}\right)\\
&=V(0,\tilde{\alpha })\otimes \left( E^{\tilde{\beta}}_{m+1,n}\oplus E^{\tilde{\beta}+1}_{m-1,n}\right)\\
&\simeq_{\cat^N}\left(V(0,\tilde{\alpha })\otimes E^{\tilde{\beta}}_{m+1,n}\right)\oplus 
\left(V(m-1, \tilde{\alpha}+1)\otimes V(n, \tilde\beta)  \right)
\end{align*}
where the isomorphism comes from the induction hypothesis.  
Combining the last two equation and using Lemma \ref{L:isomN} we see that Equation \eqref{E:gghe77} implies
$$
V(m+1, \tilde{\alpha})\otimes V(n, \tilde\beta) \simeq_{\cat^N} V(0,\tilde{\alpha })\otimes E^{\tilde{\beta}}_{m+1,n}
$$
which proves the statement for $m+1$ and concludes the induction step.  
 \end{proof}

Now we use this lemma to show $\cat^N$ has a $\bb$ map.  
  In \cite{GPT2} it is shown how to construct a $\bb$ map  from a character.  Our $\bb$ map will be defined on the representative class of simple objects $R_\A$ and extended to $\A$ by setting $\bb(W)=\bb(V)$ if $W\simeq V$ for $W\in \A$ and $V\in R_\A$.  

 Here a  \emph{character} is a map $\chi:R_{\A}\to \C$ satisfying
\begin{enumerate}
\item $\chi(V^*)=\chi(V)$ for all $ V\in R_\A$,
\item \label{I:chiOfTensorProd} if $V(m,\tilde\alpha), V(n,\tilde\beta)\in R_\A$ such that $ \bar\alpha+\bar\beta\notin \X$ then
  $$\chi(m,\tilde\alpha)\chi(n,\tilde\beta)=\sum_{k,\tilde\gamma}\dim\left(\Hom_{\cat^N}(V(k,\tilde\gamma),V(m,\tilde\alpha)\otimes V(n,\tilde\beta))\right) \chi(k,\tilde\gamma)$$  here for simplicity we denoted $\chi(V(m,\tilde\alpha))=\chi(m,\tilde\alpha)$,
\item for any $g\in G\setminus \X$, the element $\mathcal D_g=\sum_{V\in
    \cat^N_g\cap R_\A}\chi(V)^2 $ of $\C$ is non-zero.
\end{enumerate}
If $\chi$ is a character then Lemma 23 of \cite{GPT2} implies the map $G\setminus \X\to \C, \,g\mapsto \mathcal D_{g}$ is a constant function with value $\mathcal D$.  Moreover, the map $\bb=\frac1{\mathcal D}\chi$ satisfies the properties listed in Definition \ref{D:G-spherical}.  We will now show a character exists.

Let $V(m,\tilde\alpha)$ be in $ R_\A$.    Consider the formal character $\chi(m,\tilde\alpha)=\sum_{k, \tilde\gamma}c_{k,\tilde\gamma}e^ke^{\tilde\gamma}$ of $V(m,\tilde\alpha)$ in $\cat$.  Here $e^k$ and $e^{\tilde\gamma}$ are both formal variables for each $k\in \Z$ and $\tilde\gamma\in \C/\rt\Z$ and $c_{k,\tilde\gamma}$ is the dimension of the $(k,\tilde\gamma)$ weight space determined by the action of $(K_1,K_2)$. 

The variables $e^k$ and $e^{\tilde\gamma}$ of a character $\chi(m,\tilde\alpha)$ can be specialized to $q^k$ and $1$, respectively to obtain a complex number which we denote by $\chi_q(m,\tilde\alpha)\in \C$.  We will show $\chi_q$ is a character in $\cat^N$.   

Using the basis in Theorem \ref{T:Structure} we see
$$\chi_q(m,\tilde\alpha)=(2+q+q^{-1})(q^m+q^{m-2}+\cdots + q^{-m})=(2+q+q^{-1})[m+1].$$  In particular, $\chi_q(\rt'-1,\tilde\alpha)=0$.  Also, $ V(m,\tilde{\alpha } )^{\ast }=V(m, -\tilde{\alpha }-\tilde{m}-\tilde{1})$ which implies  
$\chi_q(V^*)=\chi_q(V)$ for $V\in R_\A$.

 Next, we show property \eqref{I:chiOfTensorProd} holds for $\chi_q$.  We first do this for the case $m=0 $ and   $0\leq n\leq \rt'-2$.   Let $V(0,\tilde{\alpha } ), V(n,\tilde{\beta } )\in R_\A$ such that $ \bar\alpha+\bar\beta\notin \X$.  From Lemma~\ref{L:DecompV0Vn} we have 
\begin{multline}\label{E:DecV0Vn-Cat}
V(0,\tilde{\alpha } )\otimes V(n,\tilde{\beta } )\simeq_{\cat}V(n,\tilde{\alpha }+\tilde{\beta})\oplus  V(n+1,\tilde{\alpha }+\tilde{\beta})\\
\oplus (1- \delta_{0,n})V(n-1,\tilde{\alpha }+\tilde{\beta}+1)\oplus V(n,\tilde{\alpha }+\tilde{\beta}+1)
\end{multline}
in $\cat$.  This implies
$$
\chi(0,\tilde\alpha)\chi(n,\tilde\beta)=\chi(n,\tilde{\alpha }+\tilde{\beta}) +  \chi(n+1,\tilde{\alpha }+\tilde{\beta})+(1-\delta_{0,n})\chi(n-1,\tilde{\alpha }+\tilde{\beta}+1)+\chi(n,\tilde{\alpha }+\tilde{\beta}+1).
$$
By specializing the variables of this equation we have:
\begin{multline}\label{E:DecV0Vn-chi}
\chi_q(0,\tilde\alpha)\chi_q(n,\tilde\beta)=\chi_q(n,\tilde{\alpha }+\tilde{\beta}) +\chi_q(n+1,\tilde{\alpha }+\tilde{\beta})\\
+(1-\delta_{0,n})\chi_q(n-1,\tilde{\alpha }+\tilde{\beta}+1)+\chi_q(n,\tilde{\alpha }+\tilde{\beta}+1).
\end{multline}
Translating  Equation \eqref{E:DecV0Vn-Cat} to $\cat^N$ we have: 
\begin{multline*}\label{E:DecV0Vn-CatN}V(0,\tilde{\alpha } )\otimes V(n,\tilde{\beta } )\simeq_{\cat^N}V(n,\tilde{\alpha }+\tilde{\beta})\oplus (1- \delta_{\rt'-2,n}) V(n+1,\tilde{\alpha }+\tilde{\beta})\\
\oplus (1- \delta_{0,n})V(n-1,\tilde{\alpha }+\tilde{\beta}+1)\oplus V(n,\tilde{\alpha }+\tilde{\beta}+1).
\end{multline*}
Since $\chi_q(\rt'-1,\tilde{\alpha }+\tilde{\beta})=0$ then the last equation 
implies we can rewrite Equation~\eqref{E:DecV0Vn-chi} as 
$$
\chi_q(0,\tilde\alpha)\chi_q(n,\tilde\beta)=\sum_{k,\tilde\gamma}\dim\left(\Hom_{\cat^N}(V(k,\tilde\gamma),V(0,\tilde\alpha)\otimes V(n,\tilde\beta))\right) \chi_q(k,\tilde\gamma)
$$
here all but possibly four homomorphism spaces are zero.  
This implies that if $W\simeq_{\cat^N} \bigoplus_{i=1}^s\left(V(0,\tilde\alpha)\otimes V(n_i,\tilde{\beta}_i)\right)$ where $\bar\alpha,\bar{\beta}_i,\bar\alpha+\bar{\beta}_i\notin \X$ and  $0\leq n_i\leq \rt'-2$ for all $i\in\{1,...,s\}$ then 
\begin{equation}\label{E:chi0nHom}
 \sum_{i=1}^s \chi_q(0,\tilde\alpha)\chi_q(n_i,\tilde{\beta}_i)= \sum_{k,\tilde\gamma}\dim\left(\Hom_{\cat^N}(V(k,\tilde\gamma),W)\right) \chi_q(k,\tilde\gamma).
\end{equation}

Next we consider the general case.  Let $V(m,\tilde\alpha), V(n,\tilde\beta)\in R_\A$ such that $ \bar\alpha+\bar\beta\notin \X$.  A direct computation shows
$$
\chi_q(m,\tilde\alpha)\chi_q(n,\tilde\beta)=(2+q+q^{-1})^2[m+1][n+1].
$$ where  it can be shown that 
 $$
 [m+1][n+1]=[n+1+m]+[n+1+m-2]+...+[n+1-m].
 $$
  This implies,   
\begin{multline*}
\chi_q(m,\tilde\alpha)\chi_q(n,\tilde\beta)\\
=\chi_q(0,\tilde\alpha)\left(\chi_q(n+m,\tilde\beta)+\chi_q(n+m-2,\tilde\beta+1)+...+ \chi_q(n-m,\tilde\beta+m)\right).
\end{multline*}
But from Lemma \ref{L:DecmnINTO0m+n} and Equation \eqref{E:chi0nHom} we have the right side of the last equation is equal to
$$
\sum_{k,\tilde\gamma}\dim\left(\Hom_{\cat^N}(V(k,\tilde\gamma),V(m,\tilde\alpha)\otimes V(n,\tilde\beta))\right) \chi_q(k,\tilde\gamma)
$$
 and we have shown that property   \eqref{I:chiOfTensorProd} holds.  
 
 Finally, we show $\chi_q$ satisfies the last property to be a character. 
  Fix $\tilde\alpha\in \C/\rt\Z$ such that $\bar\alpha=g\notin \X$ then for any $k\in\{0,...,\rt'-1\}$  we have 
$$ \sum_{m=0}^{\rt'-1} \chi_q(m,\tilde\alpha+k)^2 =  c^2\sum_{m=0}^{\rt'-1} [m+1]^2=\frac{c^2}{(q-q^{-1})^2}\sum_{m=0}^{\rt'-1} (q^{m+1}-q^{-m-1})^2$$
where $c=2+q+q^{-1}$.
Let us compute the  sum in this expression: 
\begin{align*}
\sum_{m=0}^{\rt'-1} (q^{m+1}-q^{-m-1})^2
&=\sum_{m=0}^{\rt'-1}\left(q^{2m+2}+q^{-2m-2}-2\right)\\
&=-2\rt'+ q^2\sum_{m=0}^{\rt'-1}q^{2m}+q^{-2}\sum_{m=0}^{\rt'-1}q^{-2m}\\
&=-2\rt'+q^2 \frac{q^{2\rt'}-1}{q^2-1}+q^{-2} \frac{q^{-2\rt'}-1}{q^{-2}-1}=-2\rt'.
\end{align*}
 
Thus, we have
$$\mathcal D_g=\sum_{V\in
    \cat^N_g\cap R_\A}\chi_q(V)^2=  \sum_{m,k=0}^{\rt'-1} \chi_q(m,\tilde\alpha+k)^2 = \sum_{k=0}^{\rt'-1}  \frac{-2\rt'c^2}{(q-q^{-1})^2}= \frac{-2(\rt')^2c^2}{(q-q^{-1})^2}$$
which is non-zero.

 In summary, $\chi_q$ is a character and as explained above leads to a map $\bb=\frac1{\mathcal D}\chi$ satisfies the properties listed in Definition \ref{D:G-spherical}.

\subsection{Main theorem}
Here we summarize the results of this paper in the following theorem.
\begin{theorem}
Let $\Gr=\C/\Z$ and $\X=\frac12\Z/\Z$. 
 Let $\A$ be the set of generic simple objects of $\cat^N$ given in  Equation \eqref{E:AsimpleV}.  Let $\qd:\A\to \C^{\times}$
 be the function defined in Equation \eqref{E:FromulaQDVnalpha22}.  Let $\bb:\A\to \FK$ be the function constructed in Subsection \ref{SS:bbmap}.  With this data $\cat^N$ is a relative $\Gr$-spherical category with basic data and leads to the modified TV-invariant described in Theorem \ref{inveee}.  
\end{theorem}
\begin{proof}
The proof follows directly from Corollary \ref{C:GenFinSSPiv} and Lemmas \ref{L:tambiPairCatN} and \ref{L:basicdata}.
\end{proof}

\end{document}